\documentclass[11pt]{article}

\usepackage{mathpazo}
\usepackage{amsfonts}
\usepackage{amsmath}
\usepackage{graphicx}
\usepackage{latexsym}
\usepackage{mathabx}

\usepackage[margin=1.05in]{geometry}
  
\usepackage[T1]{fontenc}
\usepackage{times}
\usepackage{color,graphicx}
\usepackage{array}
\usepackage{enumerate}
\usepackage{amsmath}
\usepackage{amssymb}
\usepackage{amsthm}
\usepackage{pgfplots}
\usepackage{pgf}
\usepackage{tikz}
\usetikzlibrary{patterns}
\usepgfplotslibrary{patchplots} 
\usetikzlibrary{pgfplots.patchplots} 
\pgfplotsset{width=9cm,compat=1.5.1}

\usepackage{amsthm}

\newtheorem{lemma}{Lemma}
\newtheorem{theorem}{Theorem}
\newtheorem{corollary}{Corollary}

\newtheorem{proposition}{Proposition}

\newtheorem{example}{Example}

\usepackage{xcolor}
\usepackage{makeidx}
\usepackage[colorlinks=true,linkcolor=blue,anchorcolor=blue,citecolor=red,urlcolor=magenta]{hyperref}
\usepackage[alphabetic,backrefs]{amsrefs}

\begin{document}
\title{Quantum state transfer between twins in weighted graphs}

\author{
	Stephen Kirkland,\textsuperscript{\!\!1} \ Hermie Monterde,\textsuperscript{\!\!1} \  and Sarah Plosker\textsuperscript{1,2}
}

\maketitle










\begin{abstract}
Twin vertices in simple unweighted graphs are vertices that have the same neighbours and, in the case of weighted graphs with possible loops, the corresponding incident edges have equal weights. In this paper, we explore the role of twin vertices in quantum state transfer. In particular, we provide  characterizations of periodicity, perfect state transfer, and pretty good state transfer between twin vertices in a weighted graph with respect to its adjacency, Laplacian and signless Laplacian matrices. As an application, we provide characterizations of all simple unweighted double cones on regular graphs that exhibit periodicity, perfect state transfer, and pretty good state transfer.
\end{abstract}

\noindent \textbf{Keywords:} quantum state transfer, 
twin vertices, adjacency matrix, Laplacian matrix, signless Laplacian matrix\\
	
\noindent \textbf{MSC2010 Classification:} 
05C50; 
15A18;  
05C22; 
81P45; 
81A10 


\addtocounter{footnote}{1}
\footnotetext{Department of Mathematics, University of Manitoba, Winnipeg, MB, Canada R3T 2N2}
\addtocounter{footnote}{1}
\footnotetext{Department of Mathematics \& Computer Science, Brandon University, Brandon, MB, Canada R7A 6A9}
\tableofcontents

\section{Introduction}\label{secINTRO}

\begin{figure}[h!]\label{fig:small} 
	\begin{center}
		\begin{tikzpicture}
		\tikzset{enclosed/.style={draw, circle, inner sep=0pt, minimum size=.2cm}}
	   
	   \node[enclosed, fill=cyan] (w_1) at (0,2.5) {};
		\node[enclosed, fill=cyan] (w_2) at (1.5,2.5) {};
		\draw (w_1) -- (w_2);

		\node[enclosed, fill=cyan] (v_1) at (3,3) {};
		\node[enclosed] (v_2) at (4,2) {};
		\node[enclosed, fill=cyan] (v_3) at (5,3) {};
		\draw (v_1) -- (v_2);
 		\draw (v_2) -- (v_3);
	   
		\node[enclosed, fill=cyan] (u_1) at (6.5,2.5) {};
		\node[enclosed] (u_2) at (7.5,1.5) {};
		\node[enclosed, fill=cyan] (u_3) at (8.5,2.5) {};
		\node[enclosed] (u_4) at (7.5,3.5) {};
		
		\draw (u_1) --  (u_2);
		\draw (u_1) --  (u_4);
 		\draw (u_2) -- (u_3);
 		\draw (u_3) -- (u_4);
 		
 		\node[enclosed, fill=cyan] (x_1) at (10.1,3.35) {};
		\node[enclosed] (x_2) at (10.1,1.75) {};
		\node[enclosed, fill=cyan] (x_3) at (11.75,3.35) {};
		\node[enclosed] (x_4) at (11.75,1.75) {};
		
		\draw (x_1) --  (x_2);
		\draw (x_1) --  (x_4);
 		\draw (x_2) -- (x_3);
 		\draw (x_2) -- (x_4);
 		\draw (x_3) -- (x_4);
		\end{tikzpicture}
	\end{center}
	\caption{Small unweighted graphs that exhibit perfect state transfer between vertices marked blue}\label{yay}
\end{figure}
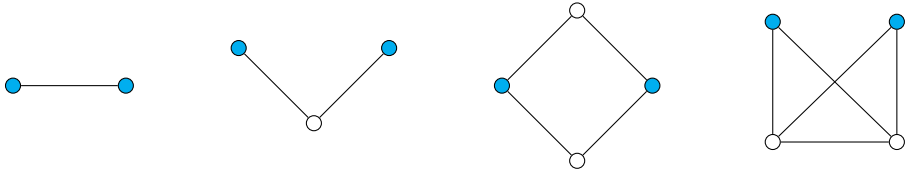

The concept of a continuous-time quantum walk was  introduced by Farhi and Gutmann \cite{Farhi1998} in 1998, but it was not until 2003 that Bose proposed the use of a continuous-time quantum walk on a path to transmit quantum states \cite{Bose:Quantum}. Motivated by high fidelity quantum state transfer, Christandl et al.\ introduced perfect state transfer \cite{Christandl:PSTonHypercubes,Christandl2005} in 2004. They showed that the path $P_n$ on $n$ vertices admits perfect state transfer only when $n=2,3$ with respect to the adjacency matrix, and $n=2$ with respect to the Laplacian matrix. This prompted researchers to search for new graphs with perfect state transfer. Some well-known examples include cubelike graphs \cite{Cheung2011}, integral circulant graphs \cite{Basic2009}, distance-regular graphs \cite{Coutinho2015} Hadamard diagonalizable graphs \cite{Johnston2017}, quotient graphs \cite{Bachman2011,Ge2011}, certain joins of graphs \cite{Angeles-Canul2009,Angeles-Canul2010}, as well as non-complete extended $p$-sums (NEPS) of some graphs \cite{Li2021,Pal2017a}. However, due to its rarity, perfect state transfer was relaxed by several authors (Godsil \cite{Godsil2012a}, Vinet and Zhedanov \cite{Vinet2012}) to what is known as pretty good state transfer, which  is ``good enough'' for physical lab setups. It turns out that $P_n$ exhibits pretty good state transfer for infinitely many $n$, as shown by Godsil et al.\ \cite{Godsil2012} and van Bommel \cite{VANBOMMEL2019} for the adjacency matrix, and Banchi et al.\ for the Laplacian matrix \cite{Banchi2017}. Pretty good state transfer was also investigated for cycles \cite{Pal2017}, a family of Cayley graphs \cite{Cao2020b}, double stars \cite{Fan2013}, and  weighted graphs with possible loops \cite{Eisenberg2019,Johnston2017,Kempton2017a}. More recently, the concept of pair state transfer was studied by Chen and Godsil \cite{Chen2019PairST}.

The graphs $K_2$ and $C_4$ are well-known examples of small graphs that exhibit perfect state transfer between antipodal vertices. It is also known that $P_3$ exhibits perfect state transfer between antipodal vertices with respect to the adjacency matrix, while $K_4\backslash e$ ($K_4$ minus an edge) admits perfect state transfer between non-adjacent vertices with respect to the Laplacian matrix. Upon examining the pairs of vertices in these graphs that admit perfect state transfer, one finds that they share the same neighbours (see Figure \ref{fig:small}). In other words, they are twins. Indeed, a number of examples of quantum state transfer in the literature can be viewed in the context of twins. For instance, there are infinitely many double cones known to exhibit adjacency and Laplacian perfect state transfer between their apexes (which are twins) \cite{Angeles-Canul2010,Alvir2016}. Kempton et al.\ also showed that adding loops with suitable weights to a pair of non-adjacent vertices with the same neighbours yields adjacency perfect state transfer between them \cite{Kempton2016}. Yet the role of twin vertices in quantum state transfer remains largely unexplored. Except for the MSc work of Monterde \cite{Monterde} and a recent paper of Pal \cite{PAL2022112872}, we are unaware of other work on this topic. In this paper, we provide a systematic approach to analyzing the properties of quantum state transfer between twin vertices in weighted graphs with possible loops. In particular, we focus on connected graphs. In the case of a disconnected graph, our results may be applied to its connected components.

In Section \ref{secTM}, we use the algebraic properties of graphs with twins to show that their transition matrices have a certain form. We then use this to provide an upper bound on the probability of state transfer between twins $u$ and $v$ in terms of $\lvert T \rvert$, where $T$ is a set of twins containing $u$ and $v$. This result then allows us to answer a question posed by Godsil in \cite{Godsil2017} in the affirmative. Section \ref{secJoins} provides some useful results on joins important for subsequent sections. In particular, we derive the transition matrix of the join of two graphs with respect to the signless Laplacian matrix. The results in this section may prove to be of independent interest in graph theory; however, we are only interested in building up the machinery to apply to our work on quantum state transfer between twins herein. The Section \ref{secPER} is devoted to a characterization of periodicity in twin vertices. We highlight our result in this section regarding the exact minimum period of a periodic vertex in a weighted graph with possible loops. In particular, we show that if the size of the eigenvalue support of a periodic vertex is at least three, then its minimum period exceeds $2\pi/(\lambda-\mu)$, where $\lambda$ and $\mu$ are the largest and smallest eigenvalues in the eigenvalue support. This result implies that the bound on the minimum period provided by Godsil \cite[Lemma 3.4]{Godsil2012a} is tight if and only if the size of the eigenvalue support is two. Section \ref{secPST} deals with perfect state transfer between twin vertices. In particular, we provide a characterization of perfect state transfer between twins, which then allows us to identify all connected and disconnected double cones on regular graphs whose apexes admit adjacency and signless Laplacian perfect state transfer. These double cones provide infinitely many examples of graphs that exhibit perfect state transfer and pretty good state transfer between twin vertices. Lastly, in Section \ref{secPGST}, we give a characterization of pretty good state transfer between twins, and again using this characterization, we identify all connected and disconnected double cones on regular graphs with non-periodic apexes that admit adjacency and signless Laplacian pretty good state transfer. The remainder of the present section is allotted to basic definitions and notation.

Some of the results presented in this paper can be found in the M.Sc.\ thesis of Monterde \cite{Monterde}. For the basics of graph theory and matrix
theory, we refer the reader to Godsil and Royle \cite{Godsil:AlgebraicGraph}, and Horn and Johnson \cite{Horn:MatrixAnalysis,horn94}, respectively. For more background on state transfer, see Godsil \cite{Godsil2012a}, and Coutinho and Godsil \cite{Coutinho2021}.

Throughout this paper, we assume that $X$ is a connected weighted undirected graph with possible loops but no multiple edges. We denote the vertex set of $X$ by $V(X)$, and we allow the edges of $X$ to have nonzero real weights (i.e., an edge can have either positive or negative weight). We say that $X$ is \textit{simple} if $X$ has no loops, and $X$ is \textit{unweighted} if all edges of $X$ have weight one. If $X$ is simple and unweighted, then we denote the complement of $X$ by $\overline{X}$, and we take the convention that $\overline{X}$ is also simple and unweighted. For $u\in V(X)$, we denote the set of neighbours of $u$ in $X$ as $N_X(u)$, and the characteristic vector of $u$ as $\textbf{e}_u$, which is a vector with a $1$ on the entry indexed by $u$ and $0$'s elsewhere. The all-ones vector of order $n$, the zero vector of order $n$, the $m\times n$ all-ones matrix, and the $n\times n$ identity matrix are denoted by $\textbf{1}_n$, $\textbf{0}_n$, $\textbf{J}_{m,n}$ and $I_n$, respectively. If $m=n$, then we write $\textbf{J}_{m,n}$ as $\textbf{J}_n$, and if the context is clear, then we simply write these matrices as $\textbf{1}$, $\textbf{0}$, $\textbf{J}$ and $I$, respectively. If $Y$ is another graph, then we write $X\cong Y$ to denote the fact that $X$ and $Y$ are isomorphic, and adopt the notation $X\vee Y$ for the join of $X$ and $Y$. We also represent the conjugate transpose of a matrix $M$ by $M^*$ and the characteristic polynomial of a square matrix $M$ in the variable $t$ by $\phi(M,t)$. Lastly, we denote the simple unweighted empty, cycle, complete, and path graphs on $n$ vertices as $O_n$, $C_n$, $K_n$, and $P_n$, resp., and the simple unweighted complete bipartite graph with partite sets of size $m$ and $n$ as $K_{m,n}$.

Two distinct vertices $u$ and $v$ of $X$ are \textit{twins} if the following conditions hold.
\begin{enumerate}
\item $N_X(u)\backslash \{u,v\}=N_X(v)\backslash \{u,v\}$.
\item The edges $(u,w)$ and $(v,w)$ have the same weight for each $w\in N_X(u)\backslash \{u,v\}$.
\item The loops on $u$ and $v$ have the same weight, and this weight is zero if those loops are absent.
\end{enumerate}
We also allow $u$ and $v$ to be adjacent, in which case $u$ and $v$ are called \textit{true twins}. Otherwise, $u$ and $v$ are called \textit{false twins}. Our definition above generalizes the definition of twin vertices from simple unweighted graphs to weighted graphs with loops.

Let $\omega,\eta\in\mathbb{R}$. A subset $T=T(\omega,\eta)$ of $V(X)$ with at least two vertices is a \textit{set of twins} in $X$ if each pair of vertices in $T$ is a pair of twins, where each vertex in $T$ has a loop of weight $\omega$, and the loops are absent if $\omega=0$, and every pair of vertices in $T$ are connected by an edge with weight $\eta$, and every pair of vertices in $T$ are not adjacent whenever $\eta=0$. Note that if $T$ is a set of twins in $X$, then either every pair of distinct vertices in $T$ are true twins, in which case $\eta\neq 0$, or every pair of distinct vertices in $T$ are false twins, in which case $\eta=0$. In particular, if $X$ is a simple unweighted graph, then $\omega=0$ and $\eta\in\{0,1\}$.

The adjacency matrix $A(X)$ of $X$ is the matrix such that
\begin{equation*}
(A(X))_{u,v}=
\begin{cases}
 \omega_{u,v}, &\text{if $u$ is adjacent to $v$}\\
 0, &\text{otherwise},
\end{cases}
\end{equation*}
where $\omega_{u,v}$ is the weight of the edge $(u,v)$. The degree matrix $D(X)$ of $X$ is the diagonal matrix of vertex degrees of $X$, where $\operatorname{deg}(u)=2\omega_{u,u}+\sum_{j\neq u}\omega_{u,j}$ for each $u\in V(X)$. The Laplacian matrix $L(X)$ of $X$ is the matrix $L(X)=D(X)-A(X)$, while the signless Laplacian matrix $Q(X)$ of $X$ is the matrix $Q(X)=D(X)+A(X)$. We use $M(X)$ to denote $A(X)$, $L(X)$ or $Q(X)$. If the context is clear, then we simply write $M(X)$, $A(X)$, $L(X)$, $Q(X)$ and $D(X)$ as $M$, $A$, $L$, $Q$ and $D$, respectively. We say that $X$ is \textit{integral} if all eigenvalues of $A(X)$ are integers, while we say that $X$ is \textit{Laplacian integral} (resp., \textit{signless Laplacian integral}) if all eigenvalues of $L(X)$ (resp., $Q(X)$) are integers.

\section{Transition Matrices}\label{secTM}

Let $X$ be a connected weighted graph with possible loops. Since $M$ is real symmetric, we can write
\begin{equation}
\label{specdec}
M=\sum_{j}\lambda_jE_j,
\end{equation} 
in its spectral decomposition, where the $\lambda_j$'s are the distinct eigenvalues of $M$ and each $E_j$ is the orthogonal projection matrix onto the eigenspace associated with $\lambda_j$. If the eigenvalues are not indexed, then we also denote by $E_{\lambda}$ the orthogonal projection matrix corresponding to the eigenvalue $\lambda$ of $M$. The matrix $M$ serves as a Hamiltonian for nearest-neighbour interactions of qubits in a quantum spin system represented by $X$. Taking the exponential of $M$ yields the following transition matrix
\begin{equation}
\label{M}
U(t)=e^{itM}
\end{equation}
of the \textit{(continuous-time) quantum walk} on $X$ with respect to $M$. Here, we flip between $t$ and $\tau$ to denote the time. Note that $U(t)$ is a complex symmetric unitary matrix, and so 
for any time $\tau$, $\sum_{j=1} \lvert U(\tau)_{u,j}\rvert ^2 =1$ for any vertex $u$ of $X$. For this reason, if $u$ and $v$ are vertices of $X$, then $\lvert U(\tau)_{u,v}\rvert^2$ is interpreted as the probability of quantum state transfer from $u$ to $v$ at time $\tau$.

Here are some important properties associated with state transfer. If $\lvert U(\tau)_{u,v}\rvert^2=1$, then we say that \textit{perfect state transfer} (PST) occurs from $u$ to $v$ at time $\tau$, and if $u=v$, then we say that $u$ is \textit{periodic} at time $\tau$. The minimum positive $\tau$ such that $\lvert U(\tau)_{u,v}\rvert^2=1$, respectively  $\lvert U(\tau)_{u,u}\rvert^2=1$, is called the minimum PST time, respectively the minimum period. If for every $\epsilon>0$, there exists $\tau$ such that $\lvert U(\tau)_{u,v}\rvert^2>1-\epsilon$, then we say that \textit{pretty good state transfer} (PGST) occurs from $u$ to $v$. Note that these various types of quantum state transfer depend on the matrix $M$, and hence, if the matrix $M$ is not specified, then the statement applies to $A$, $L$, and $Q$. We sometimes say adjacency periodicity, PST, and PGST when we talk about periodicity, PST, and PGST in the case that $M=A$; similar language applies when $M=L$ or $M=Q$. If $X$ is regular, i.e. all vertex degrees are equal,  then the quantum walks with respect to $A$, $L$, and $Q$ are equivalent, so that we get PST/PGST between $u$ and $v$ (resp., periodicity at $u$) with respect to $A$ if and only if we get PST/PGST between $u$ and $v$ (resp., periodicity at $u$) with respect to $M\in\{L,Q\}$. Meanwhile, if $X$ is bipartite, then $L$ and $Q$ are similar by a diagonal matrix of $\pm 1$'s so that 
\begin{equation}
\label{LQ}
U_L(t)_{u,v}=\pm U_Q(t)_{u,v}
\end{equation}
for any $u,v\in V(X)$. Thus, PST/PGST occurs between $u$ and $v$ (resp., periodicity) with respect to $L$ if and only if they occur with respect to $Q$. For more information about continuous-time quantum walks, see \cite{Godsil2012a,Alvir2016,Kendon2003}. For more about PST, see the survey of Kendon and Tamon \cite{Kendon2011}, Godsil \cite{Godsil2010} and Kay \cite{Kay2010}, and for periodicity, see the survey of Godsil \cite{Godsil2011}.

Making use of the fact that $f(x)=e^x$ is analytic, we can write (\ref{M}) using (\ref{specdec}) as
\begin{equation}
\label{specdecM}
U(t)=\sum_{j}e^{it\lambda_j}E_j.
\end{equation}
Let $u$ and $v$ be vertices in $X$. The \textit{eigenvalue support} of $u$ with respect to $M$, denoted $\sigma_u(M)$, is the set
\begin{equation*}
\sigma_u(M)=\{\lambda_j:E_j\textbf{e}_u\neq \textbf{0}\}.
\end{equation*}
With respect to $M$, we say that $u$ and $v$ are
\begin{enumerate}
\item \textit{cospectral} if $(E_j)_{u,u}=(E_j)_{v,v}$ for each $j$,
\item \textit{parallel} if $E_j\textbf{e}_u$ and $E_j\textbf{e}_v$ are parallel vectors for each $j$, i.e., for each $j$, there exists $c\in\mathbb{R}$ such that $E_j\textbf{e}_u=cE_j\textbf{e}_v$, and
\item \textit{strongly cospectral} if $E_j\textbf{e}_u=\pm E_j\textbf{e}_v$ for each $j$, in which case we define the sets
\begin{equation*}
\sigma_{uv}^+(M)=\{\lambda_j:E_j\textbf{e}_u=E_j\textbf{e}_v\neq \textbf{0}\}\ \text{and}\ \sigma_{uv}^-(M)=\{\lambda_j:E_j\textbf{e}_u=-E_j\textbf{e}_v\neq \textbf{0}\}.
\end{equation*}
\end{enumerate}
If $u$ and $v$ are cospectral with respect to $M$, then $\sigma_u(M)=\sigma_v(M)$ and (\ref{specdecM}) yields $U(t)_{u,u}=U(t)_{v,v}$, while if $u$ and $v$ are strongly cospectral with respect to $M$, then we get $\sigma_u(M)=\sigma_{uv}^+(M)\cup \sigma_{uv}^-(M)$. It is also known that if an automorphism maps $u$ to $v$, then they are cospectral with respect to $M$. The concepts of cospectrality, parallelism and strong cospectrality between two vertices in a graph with respect to its adjacency matrix were first studied in depth by Godsil and Smith \cite{Godsil2017}, and recently generalized to Hermitian matrices by Monterde \cite{Monterde2021}.

%


Next, we restate a spectral characterization of twin vertices with respect to $M(X)$ due to Monterde \cite[Lemma 2.9]{Monterde2021}.

\begin{lemma}
\label{alphabeta}
Let $T=T(\omega,\eta)$ be a set of twins in $X$. Then $u,v\in T$ if and only if both of the following conditions hold:
\begin{enumerate}
\item $\textbf{e}_u-\textbf{e}_v$ is an eigenvector of $M(X)$, and
\item the eigenvalue corresponding to $\textbf{e}_u-\textbf{e}_v$ is given by
\begin{equation}
\label{adjalpha}
\theta=
\begin{cases}
 \omega-\eta, &\text{if $M(X)=A(X)$}\\
 \text{deg}(u)-\omega+\eta, &\text{if $M(X)=L(X)$}\\
 \text{deg}(u)+\omega-\eta, &\text{if $M(X)=Q(X)$}.
\end{cases}
\end{equation}
\end{enumerate}
\end{lemma}

If $u$ and $v$ are twins, then Lemma \ref{alphabeta} implies that $\theta\in\sigma_u(M)$. Next, we state an algebraic characterization of twin vertices \cite[Lemma 2]{Monterde2021} as well as a property of twin vertices that is useful in quantum state transfer \cite[Lemma 2.10]{Monterde2021}.

\begin{lemma}
\label{aut}
Vertices $u$ and $v$ are twins in $X$ if and only if there exists an involution on $X$ that switches $u$ and $v$ and fixes all other vertices. Moreover, if $u$ and $v$ are twins in $X$, then $u$ and $v$ are cospectral.
\end{lemma}

Let $f$ be an automorphism of $X$, and $u$, $v$ and $w$ be vertices of $X$. Then one can easily check that $U(t)_{u,v}=U(t)_{f(u),f(v)}$ for any $t\in\mathbb{R}$. Now, if we add that $f$ fixes $w$ but sends $u$ to $v$, then we get $U(t)_{w,u}=U(t)_{f(w),f(u)}=U(t)_{w,v}$ for any $t\in\mathbb{R}$. Thus, if $O_u$ is the orbit of $u$ under $f$, then $U(t)_{w,u}=U(t)_{w,v}$ for all $v\in O_u$. Since $U(t)$ is unitary, its $w$-th row gives us
\begin{equation*}
1=\sum_{j\in V(X)}\lvert U(t)_{w,j}\rvert^2=\lvert O_u\rvert \lvert U(t)_{w,u}\rvert^2+\sum_{j\notin O_u}\lvert U(t)_{w,j}\rvert^2,
\end{equation*}
for any $t\in\mathbb{R}$. Thus, if $U(t)_{w,j}=0$ for $j\notin O_u$, then $\lvert U(t)_{w,u}\rvert^2=\frac{1}{\lvert O_u\rvert}$. We summarize this in the following proposition, a part of which was first established by Coutinho \cite[Lemma 8.1.1]{Coutinho2014}.

\begin{proposition}
\label{permut}
If $f$ is an automorphism of $X$ that fixes $w$ and $O_u$ is the orbit of $u$ under $f$, then for any $t\in\mathbb{R}$, $U(t)_{w,u}=U(t)_{w,v}$ for all $v\in O_u$, and
\begin{equation*}
\lvert U(t)_{w,u}\rvert^2\leq \frac{1}{\lvert O_u\rvert},
\end{equation*}
with equality if and only if $U(t)_{w,j}=0$ for all $j\notin O_u$.
\end{proposition}

The following result imposes a particular form on the transition matrices of graphs with twin vertices.

\begin{theorem}
\label{transmattw}
Vertices $u$ and $v$ are twins in $X$ if and only if for any $t\in\mathbb{R}$, $U(t)_{u,u}=U(t)_{v,v}$ and $U(t)_{w,u}=U(t)_{w,v}$ for all $w\in V(X)\backslash \{u,v\}$. Moreover, if $u$ and $v$ are twins in $X$, then $U(t)_{u,u}\neq U(t)_{v,u}$.
\end{theorem}

\begin{proof}
To prove necessity, let $u$ and $v$ be twin vertices in $X$. By Lemma \ref{aut}, there exists an automorphism $f$ of $X$ that switches $u$ and $v$, and fixes all other vertices. Thus for any $t\in\mathbb{R}$, Proposition \ref{permut} implies that $U(t)_{w,u}=U(t)_{w,v}$ for all $w\in V(X)\backslash\{u,v\}$. Now, by Lemma \ref{aut}, $u$ and $v$ are cospectral, and thus, $U(t)_{u,u}=U(t)_{v,v}$ for any $t\in\mathbb{R}$. To prove sufficiency, suppose $a=U(t)_{u,u}=U(t)_{v,v}$, $b=U(t)_{u,v}=U(t)_{v,u}$, and $U(t)_{w,u}=U(t)_{w,v}$ for all $w\in V(X)\backslash\{u,v\}$. A simple computation reveals that $U(t)(\textbf{e}_u-\textbf{e}_v)=(a-b)(\textbf{e}_u-\textbf{e}_v)$ so that $\textbf{e}_u-\textbf{e}_v$ is an eigenvector for $U(t)$ for any $t\in\mathbb{R}$. Consequently, $\textbf{e}_u-\textbf{e}_v$ is an eigenvector for $M$ corresponding to the eigenvalue $\theta$ given in (\ref{adjalpha}) satisfying $a-b=e^{it\theta}$. By Lemma \ref{alphabeta}, we get that $u$ and $v$ are twins in $X$. The latter statement is true because $a-b\neq 0$, otherwise the columns of $U(t)$ indexed by $u$ and $v$ are equal, i.e., $U(t)$ is singular, a contradiction.
\end{proof}

If $u$ and $v$ are twins, then Theorem \ref{transmattw} implies that $U(t)\textbf{e}_u$ and $U(t)\textbf{e}_v$ have equal entries except for those indexed by $u$ and $v$. A statement similar to Theorem \ref{transmattw} appears in \cite[Theorem 8.1.3]{Coutinho2014}, although we point out that since $U(t)_{u,u}\neq U(t)_{v,u}$, it cannot happen that $U(t)\textbf{e}_u=U(t)\textbf{e}_v$. Otherwise, $0$ is an eigenvalue of $U(t)$, which is a contradiction because $U(t)$ is unitary.

We now state a corollary to Theorem \ref{transmattw} which reveals important information about the entries of the transition matrix indexed by twin vertices.

\begin{corollary}
\label{yipee}
Let $T$ be a set of twins in $X$. For any $u,v\in T$ with $u\neq v$, $\lvert U(t)_{u,u}\rvert+ \lvert U(t)_{u,v}\rvert\geq 1$ for all $t\in\mathbb{R}$. Moreover, the following statements hold.
\begin{enumerate}
\item Vertex $u$ is periodic with period $\tau$ if and only if $(U(\tau))_{u,v}=0$.
\item If $\lvert T \rvert=2$, then perfect state transfer occurs between $u$ and $v$ at time $t$ if and only if $U(t)_{u,u}=0$, and pretty good state transfer occurs between $u$ and $v$ if and only if there exists a sequence of times $\{\tau_j\}$ such that $\displaystyle\lim_{j\rightarrow\infty}(U(\tau_j))_{u,u}=0$.
\item If $\lvert T \rvert\geq 3$ and $u\in T$, then $U(t)_{u,u}\neq 0$ for all $t\in\mathbb{R}$
\end{enumerate}
\end{corollary}

\begin{proof}
Let $u$ and $v$ be twins in $X$. From the proof of Theorem \ref{transmattw}, $U(t)_{u,u}-U(t)_{u,v}$ is an eigenvalue of $U(t)$ for any $t\in\mathbb{R}$. Using the triangle inequality and the fact that $U(t)$ is unitary, we obtain
\begin{equation}
\label{um}
1=\lvert U(t)_{u,u}-U(t)_{u,v}\rvert\leq \lvert U(t)_{u,u}\rvert+ \lvert U(t)_{u,v}\rvert.
\end{equation}
Hence, $\lvert U(t)_{u,v}\rvert=1$ if and only if $U(t)_{u,u}=0$, and $\lvert U(t)_{u,u}\rvert=1$ if and only if $U(t)_{u,v}=0$. If we further assume that $T$ is a set of twins such that $u,v\in T$ and $\lvert T \rvert\geq 3$, then Theorem \ref{transmattw} implies that $\lvert U(t)_{u,v}\rvert=\lvert U(t)_{u,w}\rvert$ whenever $w\in T\backslash\{u,v\}$. If $\lvert U(t)_{u,v}\rvert=1$, then $\lvert U(t)_{u,w}\rvert=1$ for all $w\in T\backslash\{u,v\}$, a contradiction because $U(t)$ is unitary. Thus, $\lvert U(t)_{u,v}\rvert<1$, and by (\ref{um}), $U(t)_{u,u}\neq 0$ for all $t\in\mathbb{R}$.
\end{proof}

Now, suppose $X$ has $n\geq 3$ vertices and $T$ is a set of twins in $X$. If $u,v\in T$ and we let $U(t)_{u,u}=a$ and $U(t)_{u,v}=b$, then labelling the vertices so that those in $T$ appear first, Theorem \ref{transmattw} yields
\begin{equation}
\label{eq1}
U(t)=\left[
\begin{array}{ccccc}
U_1&U_2 \\
 U_2^T&*\\
\end{array}
\right]
\end{equation}
where $U_1=(a-b)I_{\lvert T \rvert}+bJ_{\lvert T \rvert}$, $a\neq b$ and the columns of $U_2^T$ are identical. Since $U(t)$ is unitary, we have
\begin{equation}
\label{eq21}
\lvert U(t)_{u,u}\rvert^2+\left(\lvert T \rvert-1\right)\lvert U(t)_{u,v}\rvert^2+\sum_{w\notin T}\lvert U(t)_{u,w}\rvert^2=1
\end{equation}
for any $t\in\mathbb{R}$. Moreover, if $\lvert T \rvert\geq 3$, Corollary \ref{yipee}(3) implies that $\lvert U(t)_{u,u}\rvert^2> 0$, and so (\ref{eq21}) implies that
\begin{equation*}
\lvert U(t)_{u,v}\rvert^2<\frac{1}{\lvert T \rvert-1}.
\end{equation*}
These considerations yield the following result.

\begin{corollary}
\label{Lan}
Assume $X$ has $n\geq 3$ vertices and let $T$ be a set of twins in $X$. Then $U(t)$ assumes the form in (\ref{eq1}), and (\ref{eq21}) holds. Moreover, if $u,v\in T$, then $\lvert U(t)_{u,v}\rvert^2\leq \frac{1}{\lvert T \rvert-1}$ for all $t\in\mathbb{R}$, and this inequality is strict whenever $\lvert T \rvert\geq 3$.
\end{corollary}

If $X$ is connected and $T=V(X)$, then Corollary \ref{Lan} yields $U(t)=(a-b)I+b\textbf{J}$, which explains the form of the transition matrix of $K_n$.

In \cite{Godsil2017}, Godsil posed the problem: find examples of cospectral vertices $u$ and $v$ such that for some constant $\delta>0$, $\lvert U(t)_{u,v}\rvert<1-\delta$ for all $t$. We address this problem by using the above corollary. Take any graph with a set of twins $T$ with $\lvert T\rvert \geq 3$. Then any two vertices $u,v\in T$ are cospectral by Lemma \ref{aut}, and Corollary \ref{Lan} gives us $\lvert U(t)_{u,v}\rvert^2< \frac{1}{\lvert T \rvert-1}=1-\delta$ for all $t\in\mathbb{R}$, where $\delta=\frac{\lvert T\rvert -2}{\lvert T\rvert -1}$. Next, we have the following consequence of (\ref{eq21}).


\begin{corollary}
\label{nopstpgsttw}
Assume $X$ has $n\geq 3$ vertices and let $T$ be a set of twins in $X$. Then no vertex in $T$ can be involved in pretty good state transfer with a vertex that is not in $T$. Moreover, if $\lvert T \rvert\geq 3$, then any vertex in $T$ cannot be involved in pretty good state transfer with any vertex in $X$.
\end{corollary}

\section{Joins}\label{secJoins}
 
The exploration of joins with respect to quantum state transfer is not entirely new (see for instance, \cite{Angeles-Canul2009,Angeles-Canul2010}). Here, we provide a systematic approach for studying quantum state transfer between twin vertices that arise from joining either a complete or empty graph with another graph (possibly regular). The results we develop here are instrumental in completing the discussion of double cones in Section \ref{secPST}. We start by surveying the eigenvalue supports of vertices in joins of simple unweighted graphs, beginning with the adjacency matrix.

\begin{lemma}
\label{esupp11}
Let $X$ be a $k$-regular graph on $m\geq 2$ vertices and $Y$ be an $\ell$-regular graph on $n\geq 1$ vertices. Define $\lambda^{\pm}=\frac{1}{2}\left(k+\ell\pm\sqrt{D}\right)$ and $D=(k-\ell)^2+4mn$. Consider $Z=X\vee Y$, and let $u\in V(X)$ and $w\in V(Y)$. If one of $X$ and $Y$ is not complete, then $\lambda^{\pm}\in \sigma_u(A)$, $k\in \sigma_u(A)$ if and only if $X$ is disconnected, and no eigenvalue of $A(Y)$ is contained in $\sigma_u(A)$. The following also hold.
\begin{enumerate}
\item If $X=K_m$ and $Y\neq K_n$, then $\sigma_u(A)=\{\lambda^{\pm},-1\}$.
\item Let $X\neq K_m$. If $X$ is connected, then $\sigma_u(A)=\{\lambda^{\pm}\}\cup \sigma_u(A(X))\backslash\{k\}$. Otherwise, $\sigma_u(A)=\{\lambda^{\pm}\}\cup \sigma_u(A(X))$. In particular, if $X=O_m$, then $\sigma_u(A)=\{\lambda^{\pm},0\}$, and if we add that $Y=O_1$, then $\sigma_w(A)=\{\pm\sqrt{m}\}$.
\end{enumerate}
\end{lemma}

\begin{proof}
Let $\lambda_1\leq \ldots\leq \lambda_m=k$ be the eigenvalues of $A(X)$ and $\mu_1\leq \ldots\leq \mu_n=\ell$ be the eigenvalues of $A(Y)$. Using \cite[Equation (12.2.1)]{Coutinho2021}, the spectral decomposition of $A(Z)$ is
\begin{equation}
\label{esuppeq11}
A(Z)=\lambda^+E_{\lambda^+}
+\lambda^-E_{\lambda^-}
+\sum_{\lambda\neq \lambda_m}\lambda\left[ \begin{array}{ccccc} F_{\lambda}&\textbf{0} \\ \textbf{0}&\textbf{0} \end{array} \right]+\sum_{\mu\neq \mu_n}\mu\left[ \begin{array}{ccccc} \textbf{0}&\textbf{0} \\ \textbf{0}&F_{\mu} \end{array} \right],
\end{equation}
where the last two terms are absent whenever both $X$ and $Y$ are complete graphs, the last term in the above sum is absent if $n=1$, $F_k=E_k-\frac{1}{m}\textbf{J}_m$, $F_{\lambda}=E_{\lambda}$ if $\lambda<k$, and
\begin{equation*}
E_{\lambda^{\pm}}=\frac{1}{\pm m\sqrt{D}(k-\lambda^{\mp})}\left[\begin{array}{ccccc} (k-\lambda^{\mp})^2\textbf{J}_{m}&m(k-\lambda^{\mp})\textbf{J}_{m,n} \\ m(k-\lambda^{\mp})\textbf{J}_{n,m}&m^2\textbf{J}_n  \end{array} \right].
\end{equation*}
Moreover, $k$ is an eigenvalue of $A(Z)$ with orthogonal projection matrix $\left[\begin{array}{ccccc} F_{k}&\textbf{0} \\ \textbf{0}&\textbf{0} \end{array} \right]$ such that $F_k\textbf{e}_u\neq 0$ if and only if $X$ is disconnected. From these considerations, the first and second statements follow immediately. 
\end{proof}

For the Laplacian case, we have the following.

\begin{lemma}
\label{esupp22}
Let $X$ and $Y$ be graphs on $m\geq 2$ and $n\geq 1$ vertices, respectively. Consider $Z=X\vee Y$, and let $u\in V(X)$ and $w\in V(Y)$. The following hold in $Z$.
\begin{enumerate}
\item If $X=K_m$, then $\sigma_u(L)=\{0,m+n\}$.
\item Let $X\neq K_m$. If $X$ is connected, then $\sigma_u(L)=\{0,m+n,\lambda+n:0<\lambda\in\sigma_u(L(X))\}$. Otherwise, $\sigma_u(L)=\{0,m+n,n,\lambda+n:0<\lambda\in\sigma_u(L(X))\}$. In particular, if $X=O_m$, then $\sigma_u(L)=\{0,m+n,n\}$, and if we add that $Y=O_1$, then $\sigma_w(L)=\{0,m+1\}$.
\end{enumerate}
\end{lemma}

\begin{proof}
Let $0=\lambda_1\leq \ldots\leq \lambda_m$ be the eigenvalues of $L(X)$ and $0=\mu_1\leq \ldots\leq \mu_n$ be the eigenvalues of $L(Y)$. Using \cite[Equation 31]{Alvir2016}, the spectral decomposition of $L(Z)$ is given by
\begin{equation}
\label{esuppeq22}
L(Z)=\frac{1}{m+n}(0)\textbf{J}_{m,n}
+(m+n)E_{m+n}
+\sum_{\lambda\neq \lambda_1}(\lambda+n)\left[ \begin{array}{ccccc} F_{\lambda}&\textbf{0} \\ \textbf{0}&\textbf{0} \end{array} \right]+\sum_{\mu\neq \mu_1}(\mu+m)\left[ \begin{array}{ccccc} \textbf{0}&\textbf{0} \\ \textbf{0}&F_{\mu} \end{array} \right]
\end{equation}
where $E_{m+n}=\frac{1}{mn(m+n)}\left[ \begin{array}{ccccc} n^2\textbf{J}_m&-mn\textbf{J}_{m,n} \\ -mn\textbf{J}_{n,m}&m^2\textbf{J}_{n} \end{array} \right]$, the third (resp., fourth) term is absent if $X=K_m$ (resp., $Y=K_n$), the fourth term is absent if $n=1$, $F_0=E_0-\frac{1}{m}\textbf{J}_m$ and $F_{\lambda}=E_{\lambda}$ whenever $\lambda>0$. Moreover, $n$ is an eigenvalue of $L(Z)$ with orthogonal projection matrix $\left[ \begin{array}{ccccc} F_{0}&\textbf{0} \\ \textbf{0}&\textbf{0} \end{array} \right]$ such that $F_0\textbf{e}_u\neq 0$ if and only if $X$ is disconnected. From these considerations, conclusions 1-2 are straightforward.
\end{proof}

With respect to the Laplacian matrix, Alvir et al.\ calculated the transition matrix of a join of two simple unweighted graphs which are not necessarily regular \cite[Fact 8]{Alvir2016}. More recently, Coutinho and Godsil provided a similar result for the case of the adjacency matrix with the additional condition that the graphs joined are regular \cite[Lemma 12.3.1]{Coutinho2021}. For our next result, we derive the transition matrix of a join of two simple unweighted regular graphs with respect to the signless Laplacian matrix.

\begin{theorem}
\label{qjoin}
Let $X$ be a $k$-regular graph on $m\geq 2$ vertices, and $Y$ be an $\ell$-regular graph on $n\geq 1$ vertices. Let $p=2k-2\ell+n-m$, $D=(2k+2\ell+n+m)^2-8(2k\ell+km+\ell n)$, $\lambda^{\pm}=\frac{1}{2}\left(2k+2\ell+n+m\pm \sqrt{D}\right)$ and consider the matrix $E_{\lambda^{\pm}}$ in (\ref{pD}). Denote the transition matrices of $Z=X\vee Y$, $X$ and $Y$ with respect to the signless Laplacian matrix by $U_{Q}(t)$, $U_1(t)$ and $U_2(t)$, respectively. Then
\begin{equation}
\label{qtrans}
U_Q(t)=e^{it\lambda^+}E_{\lambda^+}+e^{it\lambda^-}E_{\lambda^-}+ \left[\begin{array}{ccccc} e^{itn}U_X(t)-\frac{e^{it(2k+n)}}{m}\textbf{J}_m &\textbf{0} \\ \textbf{0}&  e^{itm}U_Y(t)-\frac{e^{i(2\ell+m)}}{n}\textbf{J}_n \end{array} \right].
\end{equation}

\end{theorem}

\begin{proof}
Let $\lambda_1\leq \ldots\leq \lambda_m=2k$ be the eigenvalues of $Q(X)$ and $\mu_1\leq \ldots\leq \mu_n=2\ell$ be the eigenvalues of $Q(Y)$. If $\textbf{v}_j$ is an eigenvector of $Q(X)$ corresponding to $\lambda_j$ for $j<m$ and $\textbf{v}_j$ is orthogonal to $\textbf{1}$, then $\lambda_j+n$ is an eigenvalue of $Q(Z)$ with corresponding eigenvector $\left[ \begin{array}{ccccc} \textbf{v}_j\\ \textbf{0} \end{array} \right]$ for $j=1,\ldots,m-1$. Similarly, if $\textbf{w}_j$ is an eigenvector of $Q(Y)$ corresponding to $\mu_j$ for $j<m$ and $\textbf{w}_j$ is orthogonal to $\textbf{1}$, then $\mu_j+m$ is an eigenvalue of $Q(Z)$ with corresponding eigenvector $\left[ \begin{array}{ccccc} \textbf{0}\\ \textbf{w}_j \end{array} \right]$ for $j=1,\ldots,n-1$. Using equitable partitions, one can show that the remaining two eigenvalues of $Q(Z)$ are $\lambda^{\pm}$ with corresponding eigenvectors $\textbf{v}^{\pm }=\left[\begin{array}{ccccc} \left(p\pm \sqrt{D}\right)\textbf{1}_m\\ 2m\textbf{1}_n \end{array} \right]$. Thus, we may write
\begin{equation}
\label{wak}
Q(Z)=\lambda^+E_{\lambda^+}+\lambda^-E_{\lambda^-}+\sum_{\lambda\neq \lambda_m} (\lambda+n) \left[ \begin{array}{ccccc} F_{\lambda}&\textbf{0} \\ \textbf{0}&\textbf{0} \end{array} \right]+\sum_{\mu\neq \mu_n} (\mu+m) \left[ \begin{array}{ccccc} \textbf{0}&\textbf{0} \\ \textbf{0}&F_{\mu} \end{array} \right]
\end{equation}
where $F_{2k}=E_{2k}-\frac{1}{m}\textbf{J}_m$, $F_{\lambda}=E_{\lambda}$ whenever $\lambda<2k$, and
\begin{equation}
\label{pD}
E_{\lambda^{\pm}}=\frac{1}{m\left[\left(p\pm \sqrt{D}\right)^2+4mn\right]}\left[\begin{array}{ccccc} \left(p\pm \sqrt{D}\right)^2 \textbf{J}_m & 2m\left(p\pm \sqrt{D}\right) \textbf{J}_{m,n}\\ 2m\left(p\pm \sqrt{D}\right)\textbf{J}_{n,m}& 4m^2\textbf{J}_n \end{array} \right].
\end{equation}
Moreover, $2k$ is an eigenvalue of $Q(Z)$ with orthogonal projection matrix $\left[\begin{array}{ccccc} F_{2k}&\textbf{0} \\ \textbf{0}&\textbf{0} \end{array} \right]$ such that $F_{2k}\textbf{e}_u\neq 0$ if and only if $X$ is disconnected. Using (\ref{specdecM}), (\ref{wak}), and the fact that $U_X(t)=\frac{e^{i2kt}}{m}\textbf{J}_m+\sum_{\lambda\neq 2k}e^{it\lambda}F_{\lambda}$ completes the proof.
\end{proof}


We note that the last two terms in (\ref{qtrans}) are absent whenever both $X$ and $Y$ are complete and the last term in the above sum is absent whenever $n=1$.

Finally, we state the following corollary about eigenvalue supports of vertices in a join with respect to the signless Laplacian matrix. The proof is similar to Lemma \ref{esupp11}.

\begin{corollary}
\label{esupp33}
Assume the hypothesis of Theorem \ref{qjoin} holds. If one of $X$ and $Y$ is not complete, then $\lambda^{\pm}\in \sigma_u(Q)$, $2k\in \sigma_u(Q)$ if and only if $X$ is disconnected, and no eigenvalue of $Q(Y)$ is contained in $\sigma_u(Q)$. The following also hold in $Z$.
\begin{enumerate}
\item If $X=K_m$ and $Y\neq K_n$, then $\sigma_u(Q)=\{\lambda^{\pm},m+n-2\}$.
\item Let $X\neq K_m$. If $X$ is connected, then $\sigma_u(Q)=\{\lambda^{\pm}\}\cup\sigma_u(Q(X))\backslash\{2k\}$. Otherwise, $\sigma_u(Q)=\{\lambda^{\pm}\}\cup\sigma_u(Q(X))$. In particular, if $X=O_m$, then $\sigma_u(Q)=\{\lambda^{\pm},n\}$.
\end{enumerate}
\end{corollary}

\section{Periodicity}\label{secPER}

Denote the minimum period of a periodic vertex $u$ by $\rho$. Then every period $\tau$ of $u$ is an integer multiple of $\rho$. Moreover, if PST occurs between $u$ and $v$ at time $\tau$, then both of them are periodic at time $2\tau$. The converse of this is not necessarily true as periodic vertices need not exhibit PST. However, it is shown by Godsil in \cite{Godsil2012a} that if $u$ is periodic and there is PST between $u$ and $v$, then the minimum PST time between $u$ and $v$ is $\rho/2$, and PST occurs between $u$ and $v$ at every odd multiple of $\rho/2$. Now, since periodicity is a necessary condition for PST, to characterize PST between twin vertices, we first need to characterize periodic twin vertices. To do this, we state the famous Ratio Condition due to Godsil \cite{Godsil2011}.

\begin{theorem}
\label{ratiocon}
Let $X$ be a weighted graph, possibly with  loops. The following are equivalent. 
\begin{enumerate}
\item Vertex $u$ of $X$ is periodic.
\item For all $\lambda_p,\lambda_q,\lambda_r,\lambda_s\in\sigma_u(M)$ with $\lambda_r\neq \lambda_s$, we have
\begin{equation}
\label{rc}
\dfrac{\lambda_p-\lambda_q}{\lambda_r-\lambda_s}\in\mathbb{Q}.
\end{equation}
\end{enumerate}
If in addition we assume that $\phi(M,t)\in \mathbb{Z}[x]$, then $u$ is periodic if and only if either (i) $\sigma_u(M)\subseteq \mathbb{Z}$, or (ii) every $\lambda_j\in \sigma_u(M)$ is of the form $\lambda_j=\frac{1}{2}\left(a+c_j\sqrt{\Delta}\right)$, where $a$, $c_j$ and $\Delta>1$ are integers such that $\Delta$ is square-free, and the difference between any two eigenvalues in $\sigma_u(M)$ is an integer multiple of $\sqrt{\Delta}$.
\end{theorem}

Eigenvalue supports are known to contain at least two elements \cite[Proposition 2.8]{Monterde2021}. The following result determines the minimum period of periodic vertices.

\begin{theorem}
\label{minperiod}
Let $u$ and $v$ be vertices in $X$, and $\sigma_u(M)=\{\lambda_1,\lambda_2,\ldots,\lambda_n\}$ with $ \lambda_1>\lambda_2$.
\begin{enumerate}
\item If $\lvert\sigma_u(M)\rvert=2$, then $u$ is periodic with $\rho=\frac{2\pi}{\lambda_1-\lambda_2}$.
\item If $\lvert\sigma_u(M)\rvert\geq 3$ and $u$ is periodic, then $\rho=\frac{2\pi q}{\lambda_1-\lambda_2}$, where $q=\operatorname{lcm}(q_2,\ldots,q_n)$ and each $q_j$ is an integer such that $\frac{\lambda_1-\lambda_j}{\lambda_1-\lambda_2}=\frac{p_j}{q_j}$ for some integer $p_j$ such that $\text{gcd}(p_j,q_j)=1$.
\end{enumerate}
\end{theorem}

\begin{proof}
The first statement is straightforward, and so we only prove the second. Let us suppose that $u$ is periodic with $\sigma_u(M)=\{\lambda_1,\ldots,\lambda_n\}$ for some $n\geq 3$. Then Theorem \ref{ratiocon} holds, and we may let $\frac{\lambda_1-\lambda_j}{\lambda_1-\lambda_2}=\frac{p_j}{q_j}$, where each $p_j$ and $q_j$ are integers such that $\text{gcd}(p_j,q_j)=1$. The fact that $u$ is periodic is equivalent to the existence of a time $t$ and unit $\gamma\in\mathbb{C}$ such that
\begin{equation*}
U(t)\textbf{e}_u=\gamma \textbf{e}_u.
\end{equation*}
The spectral decomposition of $U(t)$ in (\ref{specdecM}) allows us to write the above equation as $\sum_je^{it\lambda_j}E_je_u=\sum_j\gamma E_je_u$. Equivalently, $e^{it(\lambda_1-\lambda_j)}=1$, which holds if and only if for each $j\geq 2$, we have $t(\lambda_1-\lambda_j)=2k\pi$ for some integer $k$. Now, suppose $\rho=\frac{2\pi z}{\lambda_1-\lambda_2}$ for some $z\in\mathbb{R}$. Since $\text{gcd}(p_j,q_j)=1$ for each $j\geq 2$, we get that $\rho(\lambda_1-\lambda_j)=2\pi z\left(\frac{\lambda_1-\lambda_j}{\lambda_1-\lambda_2}\right)=\frac{2\pi zp_j}{q_j}$ is an integer multiple of $2\pi$ if and only if $z$ is the minimum integer such that each $q_j$ divides $z$. Therefore, $z=q$, where $q=\operatorname{lcm}(q_2,\ldots,q_n)$ and so $\rho=\frac{2\pi q}{\lambda_1-\lambda_2}$.
\end{proof}

If $\lambda_1$ and $\lambda_2$ are the largest and smallest eigenvalues in $\sigma_u(M)$, then Godsil showed that $\rho\geq \frac{2\pi}{\lambda_1-\lambda_2}$ (see \cite[Lemma 3.4]{Godsil2012a}). However, if $\lvert\sigma_u(M)\rvert\geq 3$, then $q$ in Theorem \ref{minperiod}(2) satisfies $q>1$, and so $\rho>\frac{2\pi}{\lambda_1-\lambda_2}$. Thus, the inequality $\rho\geq \frac{2\pi}{\lambda_1-\lambda_2}$ is tight if and only if $\lvert\sigma_u(M)\rvert=2$. We also note that while $q$ depends on the choice of $\lambda_1$ and $\lambda_2$, the minimum period $\rho$ in Theorem \ref{minperiod} does not.

%
%

We say that a subset $W\subseteq V(X)$ is \textit{periodic} if each vertex in $W$ is periodic and there exists a time $\tau>0$ such that $\lvert U(\tau)_{u,u}\rvert=1$ for each $u\in W$. The minimum time such that $W$ is periodic is called the \textit{minimum period} $\rho$ of $W$. In particular, if $W=V(X)$, then we say that $X$ is periodic. As each vertex in $W$ is periodic, it follows that $\rho$ is an integer multiple of the minimum periods of the vertices in $W$. It is also immediate that if all vertices in $W$ have the same eigenvalue support, then $W$ is periodic if and only if one of its vertices is periodic, and as a consequence, each vertex has the same minimum period $\rho,$ which equals the minimum period of $W$.

Now, if $T$ is a set of twins in $X$, then Lemma \ref{aut} implies that the vertices in $T$ are pairwise cospectral, and thus, they all have the same eigenvalue support. Combining this with Lemma \ref{alphabeta} and \cite[Theorem 6.1]{Godsil2010} yields a characterization of periodic twin vertices whenever $\phi(M,t)\in \mathbb{Z}[x]$.

\begin{theorem}
\label{pertwin}
Let $\phi(M,t)\in \mathbb{Z}[x]$ and $T$ be a set of twins in $X$ with $\sigma_u(M)=\{\theta,\lambda_1,\ldots,\lambda_r\}$ for each $u\in T$. Then $T$ is periodic if and only if $\lambda_j=\theta+b_j\sqrt{\Delta}$ for each $j$, where $\theta$ is given in (\ref{adjalpha}), $b_j$ is an integer, and either $\Delta=1$ or $\Delta>1$ is a square-free integer. Moreover, if $T$ is periodic, then each $u\in T$ has minimum period $\rho=2\pi/g\sqrt{\Delta}$, where $g=\operatorname{gcd}(b_1,\ldots,b_r)$.
\end{theorem}

\begin{example}
Let $n\geq 2$ and consider the star $K_{1,n}\cong O_n\vee O_1$ with set of leaves $T=\{u_1,\ldots,u_n\}$. Then $T$ is a set of twins in $K_{1,n}$ and Lemma \ref{esupp11}(2b) yields $\sigma_{u_j}(A)=\{\theta,\pm\sqrt{n}\}$ for each $j$, where $\theta=0$. Moreover, Lemma \ref{esupp22}(2) gives us $\sigma_{u_j}(L)=\{1,0,n+1\}\subseteq\mathbb{Z}$. Invoking Theorem \ref{pertwin} and (\ref{LQ}), we conclude that $T$ is periodic with respect to $M\in\{A,L,Q\}$.
\end{example}

We now use Theorem \ref{pertwin} to characterize adjacency periodic twin vertices in simple unweighted joins of the form $X\vee Y$, where $X$ is either $K_m$ or $O_m$, and $Y$ is regular.

\begin{theorem}
\label{join11}
Let $m\geq 2$, $X\in\{K_m,O_m\}$, and $Y$ be an $\ell$-regular graph on $n\geq 1$ vertices. Consider $\lambda^{\pm}$ and $D$ in Lemma \ref{esupp11}, where $k=m-1$ whenever $X=K_m$ and $k=0$ whenever $X=O_m$. Suppose $Z=X\vee Y$, and let $T=V(X)$ and $S=V(Y)$. The following hold in $Z$.
\begin{enumerate}
\item If $X=K_m$ and $Y=K_n$, then $Z$ is adjacency periodic with $\rho=\frac{2\pi}{m+n}$.
\item Let $Y\neq K_n$. Then $T$ is adjacency periodic if and only if either
\begin{enumerate}
\item $X=O_m$ and $Y=O_n$, in which case $\rho=\frac{2\pi}{\sqrt{mn}}$, or
\item $D$ is a perfect square, in which case $\rho=\frac{2\pi}{g}$, where $g=\operatorname{gcd}(\lambda^{-}-\theta, \lambda^{+}-\theta)$, $\theta=-1$ if $X=K_m$, and $\theta=0$ if $X=O_m$.
\end{enumerate}
Moreover, if $Y$ is disconnected and $n\geq 2$, then $S$ is adjacency periodic if and only if $D$ is a perfect square and $Y$ is integral.
\end{enumerate}
\end{theorem}

\begin{proof}
Since 1 is clear, it suffices to show 2. Let $Y\neq K_n$, and suppose $u\in T$ and $w\in S$. By Lemma \ref{esupp11}(2), we have $\sigma_u(A)=\{\lambda^{\pm},\theta\}$, where $\theta=-1$ if $X=K_m$ and $\theta=0$ if $X=O_m$, and we can write $\lambda^{\pm}=\theta+\frac{1}{2}\left(k+\ell-2\theta\pm\sqrt{D}\right)$, where $k+\ell-2\theta\geq 0$ and $k=m-1$ whenever $X=K_m$ and $k=0$ whenever $X=O_m$. Invoking Theorem \ref{pertwin}, we conclude that $T$ is periodic if and only if either $k+\ell-2\theta=0$ or $D$ is a perfect square. The former is only possible if $k=\ell=\theta=0$, in which case $X=O_m$ and $Y=O_n$ so that $\sigma_u(A)=\{0,\pm\sqrt{mn}\}$ by Lemma \ref{esupp11}(2b), and thus, $\rho=\frac{2\pi}{\sqrt{mn}}$ by Theorem \ref{minperiod}(2). For the latter case, we get $\sigma_u(A)\subseteq \mathbb{Z}$, and so by Theorem \ref{ratiocon}, $T$ is periodic, and Theorem \ref{minperiod} gives us $\rho=\frac{2\pi}{g}$, where $g=\operatorname{gcd}(\lambda^{-}-\theta, \lambda^{+}-\theta)$. Finally, if $Y$ is disconnected and $n\geq 2$, then $\ell,\lambda^{\pm}\in\sigma_w(A)$ by Lemma \ref{esupp11}(2), and so Theorem \ref{ratiocon} implies that $S$ is adjacency periodic if and only if $D$ is a perfect square and $Y$ is integral. We note that this result about periodicity of $S$ in $Z$ holds for any regular graph $X$, and not just for $X\in\{K_m,O_m\}$.
\end{proof}

For the Laplacian case, the following is a direct consequence of Lemma \ref{esupp22}.

\begin{theorem}
\label{join22}
Let $m\geq 2$, $X\in\{K_m,O_m\}$, and $Y$ be a graph on $n\geq 1$ vertices. Consider $Z=X\vee Y$, and let $T=V(X)$ and $S=V(Y)$. The following hold in $Z$.
\begin{enumerate}
\item $T$ (resp., $S$ whenever $n\geq 2$) is Laplacian periodic with $\rho=\frac{2\pi}{m+n}$ whenever $X=K_m$ (resp., $Y=K_n$), while $\rho=\frac{2\pi}{g}$ whenever $X=O_m$ (resp., $Y=O_n$), where $g=\operatorname{gcd}(m,n)$. Moreover, if $n=1$, then $S$ is Laplacian periodic with $\rho=\frac{2\pi}{m+1}$.
\item Let $Y\notin\{K_n,O_n\}$. Then $S$ is Laplacian periodic if and only if $Y$ is Laplacian integral.
\end{enumerate}
\end{theorem}

Lastly, we deal with the signless Laplacian matrix.

\begin{theorem}
\label{join33}
Let $X$ be either $K_m$ or $O_m$ with $m\geq 2$, and $Y$ be an $\ell$-regular graph on $n\geq 1$ vertices. Consider $\lambda^{\pm}$ and $D$ in Theorem \ref{qjoin}, where $k=m-1$ whenever $X=K_m$ and $k=0$ whenever $X=O_m$. Suppose $Z=X\vee Y$ where $Y\neq K_n$, and let $T=V(X)$ and $S=V(Y)$. The following hold in $Z$.
\begin{enumerate}
\item If $X=K_m$, then $T$ is signless Laplacian periodic if and only if either $n=m+2\ell+2$ or $D$ is a perfect square.
\item If $X=O_m$, then $T$ is signless Laplacian periodic if and only if either $\ell=0$, $n=2\ell+m$ or $D$ is a perfect square.
\item If $Y$ is disconnected and $n\geq 2$, then $S$ is signless Laplacian periodic if and only if $D$ is a perfect square and $Y$ is signless Laplacian integral.
\end{enumerate}
Moreover, the minimum period of $T$ is $\rho=\frac{2\pi q}{\sqrt{D}}$, where $p$ and $q$ are integers with $\operatorname{gcd}(p,q)=1$ such that $\frac{\lambda^+-\theta}{\lambda^+-\lambda^-}=\frac{p}{q}$, and $\theta=m+n-2$ if $X=K_m$ and $\theta=n$ if $X=O_m$.
\end{theorem}

\begin{proof}
Since the case $X=K_m$ and $Y=K_n$ is equivalent to Theorem \ref{join11}(1), it suffices to assume that $Y$ is not complete. Let $u\in T$ and $w\in S$. From Corollary \ref{esupp33}(2), we have $\sigma_u(Q)=\{\lambda^{\pm},m+n-2\}$ whenever $X=K_m$, while $\sigma_u(Q)=\{\lambda^{\pm},n\}$ whenever $X=O_m$. If $X=K_m$, then $\lambda^{\pm}=\theta+\frac{1}{2}(m+2\ell-n+2+\sqrt{D})$, where $\theta=m+n-2$. On the other hand, if $X=O_m$, then we can write $\lambda^{\pm}=\theta+\frac{1}{2}(2\ell-n+m\pm \sqrt{D})$, where $\theta=n$. Invoking Theorem \ref{pertwin} and the fact that the case $\ell=0$ whenever $X=O_m$ is equivalent to Theorem \ref{join11}(2) proves 1 and 2, and Theorem \ref{minperiod}(2) yields the corresponding minimum periods. To prove 3, we note from Corollary \ref{esupp33}(2) that $2k\in\sigma_w(Q)$. By Theorem \ref{ratiocon}, $S$ is periodic if and only if $D$ is a perfect square and $Y$ is signless Laplacian integral. Again, we note that this result about periodicity of $S$ in $Z$ holds for any regular graph $X$, and not just for $X\in\{K_m,O_m\}$.
\end{proof}

Theorems \ref{join11}, \ref{join22}, and \ref{join33} provide a plethora of join graphs that exhibit periodicity with respect to their adjacency, Laplacian, or signless Laplacian matrix, respectively.

\section{Perfect state transfer}\label{secPST}

We first state an important observation due to Dave Morris \cite[Lemma 13.1]{Godsil2012a}.

\begin{lemma}
\label{pgstsc}
If pretty good state transfer occurs between $u$ and $v$, then $u$ and $v$ are strongly cospectral.
\end{lemma}

It is well-known that two vertices are strongly cospectral if and only if they are cospectral and parallel. Since twin vertices are cospectral by Lemma \ref{aut}, it follows that twin vertices are strongly cospectral if and only if they are parallel. Next, we state the Corollaries 3.10 and 3.14 in \cite{Monterde2021} respectively, which will prove useful in this section.

\begin{lemma}
\label{3tw}
Let $T$ be a set of twins in $X$. If $\lvert T\rvert \geq 3$, then each vertex $v\in T$ is not parallel, and hence not strongly cospectral, with any vertex $z\neq u$.
\end{lemma}



\begin{lemma}
\label{strcospchar}
Let $T=\{u,v\}$ be a set of twins in $X$, and consider $\theta$ in (\ref{adjalpha}). If $\Omega$ is an orthogonal set of eigenvectors for $\theta$ such that $\textbf{e}_u-\textbf{e}_v\in \Omega$, then $u$ and $v$ are strongly cospectral if and only if (i) $\lvert\Omega\rvert=1$ or (ii) $\textbf{w}^T\textbf{e}_u=\textbf{w}^T\textbf{e}_v=0$ for all $\textbf{w}\in\Omega\backslash\{\textbf{e}_u-\textbf{e}_v\}$. Moreover, if $u$ and $v$ are strongly cospectral, then $\sigma_{uv}^-(M)=\{\theta\}$, and $u$ and $v$ cannot be strongly cospectral to any $w\in V(X)\backslash\{u,v\}$.
\end{lemma}

If $u$ and $v$ are are strongly cospectral, then $\lvert\sigma_u(M)\rvert\geq 3$ \cite[Theorem 3.4]{Monterde2021}. Thus, if $u$ and $v$ are twins that are strongly cospectral, then Lemma \ref{strcospchar} yields $\lvert \sigma_{uv}^+(M)\rvert\geq 2$ and $\lvert \sigma_{uv}^-(M)\rvert=1$. A related result of Coutinho and Liu says that if $u$ and $v$ are strongly cospectral and $\lvert \sigma_{uv}^-(M)\rvert=1$, then $u$ and $v$ are twins \cite[Lemma 3.1]{Coutinho2014a}. Thus, if $u$ and $v$ are strongly cospectral with respect to $M$, then $u$ and $v$ are twins if and only if $\lvert \sigma_{uv}^-(M)\rvert=1$.

It is known that PST is monogamous \cite{Kay2011}. However, PGST is not, as shown by the Cartesian product of $P_2$ and $P_3$ provided by Pal and Bhattacharjya in \cite[Example 4.1]{PAL2017746} which exhibits pairwise adjacency PGST between four vertices. For the weighted case, Johnston et al.\ provided a graph that exhibits Laplacian PGST from one vertex to three distinct vertices \cite[Example 2]{Johnston2017}. However, from Lemma \ref{3tw}, if PGST occurs between vertices in a set of twins $T$, then $\lvert T\rvert=2$. Combining this with Lemma \ref{strcospchar}, we conclude that a vertex $u$ with a twin in $X$ can only pair up with at most one vertex $v$ to exhibit PGST. That is, PGST is monogamous when it involves a vertex with a twin.

Making use of Lemma \ref{strcospchar} and a characterization of PST by Coutinho \cite[Theorem 2.4.4]{Coutinho2014}, we obtain the following characterization of PST between twin vertices.

\begin{theorem}
\label{genminpst}
Let $T=\{u,v\}$ be a set of twins in $X$ and suppose $\sigma_u(M)=\{\theta,\lambda_1,\ldots,\lambda_r\}$, where $\theta$ is given in (\ref{adjalpha}). Then perfect state transfer occurs between $u$ and $v$ if and only if
\begin{enumerate}
\item $E_{\theta}\textbf{e}_u=-E_{\theta}\textbf{e}_v$, $E_j\textbf{e}_u=E_j\textbf{e}_v$ $j=1,\ldots,r$; and
\item there exists a time $\tau$ such that for each $j=1,\ldots,r$, an odd $m_j$ exists such that
\begin{equation}
\label{haha1}
\tau(\lambda_j-\theta)=m_j\pi.
\end{equation}
\end{enumerate}
In addition, the minimum time that perfect state transfer occurs between $u$ and $v$ is $\tau=\frac{\pi q}{\lambda_1-\lambda_2}$, where $q$ is an integer given in Theorem \ref{minperiod}.
\end{theorem}

Denote the largest power of two that divides an integer $b$ by $\nu_2(b)$. With the assumption in Theorem \ref{genminpst}, we further suppose that $\phi(M,t)\in \mathbb{Z}[x]$. Applying the well-known characterization of PST due to Coutinho \cite[Theorem 2.4.4]{Coutinho2014} for the case that $\phi(M,t)\in \mathbb{Z}[x]$, we obtain the a characterization of perfect state transfer between twin vertices whenever $\phi(M,t)\in \mathbb{Z}[x]$.

\begin{theorem}
\label{pstchartw}
Let $\phi(M,t)\in\mathbb{Z}[x]$ and $T=\{u,v\}$ be a set of twins in $X$ with $\sigma_u(M)=\{\theta,\lambda_1,\ldots,\lambda_r\}$, where $\theta$ is given in (\ref{adjalpha}). Then perfect state transfer occurs between $u$ and $v$ if and only if the following conditions holds.
\begin{enumerate}
\item $u$ and $v$ are strongly cospectral with $\sigma_{uv}^+(M)=\{\lambda_1,\ldots,\lambda_r\}$ and $\sigma_{uv}^-(M)=\{\theta\}$.
\item For each $j$, $\lambda_j=\theta+b_j\sqrt{\Delta}$, where $b_j$ is an integer, and $\Delta=1$ or $\Delta>1$ is a square-free integer.
\item For each $j$, $\nu_2(b_j)=q$, where $q$ is a nonnegative integer.
\end{enumerate}
In addition, if perfect state transfer occurs between $u$ and $v$, then the minimum PST time  is $\tau=\frac{\pi}{g\sqrt{\Delta}}$, where $g=\operatorname{gcd}(b_1,\ldots,b_r)$.
\end{theorem}

In Theorem \ref{pstchartw}, conditions (1) and (2) respectively reflect the fact that strong cospectrality and periodicity are necessary conditions for PST. To check strong cospectrality between twins, one may use Lemma \ref{strcospchar}. We also note that the minimum PST time in Theorem \ref{pstchartw} is indeed half of the minimum period indicated in Theorem \ref{pertwin}. Lastly, we remark that Theorem \ref{pstchartw} can be proven using Theorems \ref{pertwin} and \ref{genminpst}.

We illustrate Theorem~\ref{pstchartw} using the cocktail party graph $\overline{mK_2}$ as an example.
 
\begin{example}
Note that $\overline{mK_2}$ contains $m$ pairs of false twins and is $(2m-2)$-regular. The eigenvalues of $A$ are $2m-2$, $\theta=0$ (multiplicity $m$), and $-2$ (multiplicity $m-1$). One checks that any pair of false twins in $\overline{mK_2}$ are strongly cospectral with eigenvalue support containing all three eigenvalues of $A$. By Theorem \ref{pertwin}, every vertex of $\overline{mK_2}$ is periodic with minimum period $\rho=\pi$. Moreover, since the largest power of two that divides $2(m-1)$ and $-2$ are equal if and only if $m-1$ is odd, Theorem \ref{pstchartw} yields PST between any pair of false twins in $\overline{mK_2}$ if and only if $m$ is even, in which case the minimum PST time is $\tau=\pi/2$.
These observations also apply to the Laplacian and signless Laplacian case because $\overline{mK_2}$ is regular.
\end{example}

Let $X\in\{O_m,K_m\}$ and $Y$ be a graph on $n\geq 1$ vertices. In the simple unweighted join $X\vee Y$, the vertices in $X$ form a set of twins $T$ in $X\vee Y$. If $m\geq 3$, then Corollary \ref{nopstpgsttw} implies that any vertex in $X$ cannot be involved in PGST with any other vertex in $X\vee Y$. This motivates us to look at the case $m=2$. Let $m=2$ and $V(X)=\{u,v\}$. The join $X\vee Y$ is called a \textit{double cone} on $Y$ with apexes $u$ and $v$. In particular, if $X=K_2$, then we call $X\vee Y$ a \textit{connected double cone} on $Y$. Otherwise, $X\vee Y$ is a \textit{disconnected double cone} on $Y$. Earlier work by Angeles-Canul et al.\ provides partial results for adjacency PST in double cones on regular graphs \cite[Corollaries 13, 15]{Angeles-Canul2010}. 

Here, we go a step further by providing a complete characterization of double cones on regular graphs which exhibit adjacency and signless Laplacian PST. This characterization highlights infinite families of graphs having adjacency or signless Laplacian PST.  We begin with the disconnected case.

\begin{theorem}
\label{discdc}
Let $Y$ be a $\ell$-regular graph on $n\geq 1$ vertices. The following hold.
\begin{enumerate}
\item Adjacency perfect state transfer occurs between the apexes of $O_2\vee Y$ if and only if either (i) $\ell=0$ or (ii) $\ell>0$, $n=\frac{1}{2}s(\ell+s)$ for some integer $s$, and $\nu_2(\ell)>\nu_2(s)\geq 1$. Moreover, if $\ell=0$, then the minimum time that perfect state transfer occurs is $\tau=\frac{\pi}{\sqrt{2n}}$. Otherwise, it is $\tau=\frac{\pi}{g}$, where $g=\operatorname{gcd}(\ell+s,s)$.
\item Signless Laplacian perfect state transfer occurs between the apexes of $O_2\vee Y$ if and only if one of the following conditions holds
\begin{enumerate}
\item $\ell=0$ and $n\equiv 2$ (mod 4);
\item $\ell>0$ and $n=2\ell+2$; or
\item $\ell>0$, $n=\frac{s(2\ell-s+2)}{2\ell-s}$ for some integer $s$, and either (i) $\nu_2(s)>1$, $\ell$ is even and $\nu_2(n)=1$ or (ii) $\ell$ is odd, $\nu_2(n)>\nu_2(s)$ and $\nu_2(\ell+1)>\nu_2(s)-1$.
\end{enumerate}
Moreover, the minimum time that perfect state transfer occurs is $\tau=\frac{\pi}{2}$ whenever $\ell=0$, and $\tau=\frac{\pi}{\sqrt{2n}}$ whenever $\ell>0$ and $n=2\ell+2$. Otherwise, it is $\tau=\frac{\pi}{g}$, where $g=\operatorname{gcd}(2\ell-s+2,n-s)$.
\end{enumerate}
\end{theorem}

\begin{proof}
Let $u$ and $v$ be the apexes of $O_2\vee Y$. By virtue of \cite[Corollary 6.9(1)]{Monterde2021}, $u$ and $v$ are strongly cospectral both with respect to $A$ and $Q$. Making use of Lemmas \ref{alphabeta} and \ref{strcospchar}, we get that $\sigma_{uv}^-(A)=\{0\}$ and $\sigma_{uv}^-(Q)=\{n\}$. We divide our discussion into two cases: the adjacency case and the signless Laplacian case.

We begin by proving the first statement. By Lemma \ref{esupp11}(2), we get $\sigma_u(A)=\{\lambda^{\pm},0\}$, where $\lambda^{\pm}=\frac{1}{2}\left(\ell\pm\sqrt{\ell^2+8n}\right)$. If $\ell=0$, then $\sigma_u(A)=\left\{\pm\sqrt{2n},0\right\}$, where $\nu_2(\sqrt{2n})=\nu_2(-\sqrt{2n})$. By Theorem \ref{pstchartw}, PST occurs between $u$ and $v$ with minimum time $\tau=\frac{\pi}{\sqrt{2n}}$, and so (i) holds. Now, let $\ell>0$. By Theorem \ref{pstchartw}, adjacency PST occurs between $u$ and $v$ if and only if $\ell^2+8n$ is a perfect square and
\begin{equation}
\label{Anu1}
\nu_2(\ell+\sqrt{\ell^2+8n})=\nu_2(\ell-\sqrt{\ell^2+8n}).
\end{equation}
Now, $\ell^2+8n$ is a perfect square if and only if $8n=4s(\ell+s)$ for some integer $s$ such that $s(\ell+s)$ is even. Making use of (\ref{Anu1}), we get $\nu_2(\ell+s)=\nu_2(s)$ which implies that $\nu_2(\ell)>\nu_2(s)$. As $\nu_2(s)\geq 0$ and $\nu_2(n)=\nu_2\left(\frac{s(\ell+s)}{2}\right)=\nu_2(s)+\nu_2(\ell+s)-1=2\nu_2(s)-1\geq 0$, we obtain $\nu(s)\geq 1$. Therefore, (ii) is true. Lasly, invoking Theorem \ref{pstchartw} yields the minimum PST time $\tau=\frac{\pi}{2g}$, where $g=\operatorname{gcd}(\ell+\sqrt{\ell^2+8n},\ell-\sqrt{\ell^2+8n})$.

Next, we show the second statement. By Corollary \ref{esupp33}(2b), we get $\sigma_u(Q)=\{\lambda^{\pm},\theta\}$, where $\lambda^{\pm}=\frac{1}{2}(2\ell+n+2\pm\sqrt{(2\ell+n+2)^2-8\ell n})$ and $\theta=n$ in Theorem \ref{pstchartw}. We have the following cases.
\begin{itemize}
\item \small{Let $\ell=0$ so that $\sigma_u(Q)=\{n,n+2,0\}$. Since we can write $n+2=\theta+2$ and $0=\theta+(-n)$, we obtain $\nu_2(-n)=\nu_2(2)=1$ if and only if $n=2c$ for some odd $c$, or equivalently, $n\equiv 2$ (mod 4). Applying Theorem \ref{pstchartw}, we get signless Laplacian PST between $u$ and $v$ if and only if $n\equiv 2$ (mod 4), in which case the minimum PST time is $\tau=\frac{\pi}{2}$.}
\item Let $\ell>0$ so that we can write $\lambda^{\pm}=\theta+\frac{1}{2}\left(2\ell-n+2\pm\sqrt{(2\ell+n+2)^2-8\ell n}\right)$. Invoking Theorems \ref{join33}(2) and \ref{pstchartw}(2), we need either $n=2\ell+2$ or $(2\ell+n+2)^2-8\ell n$ is a perfect square. If $n=2\ell+2$, then $\lambda^{\pm}=\theta\pm 2\sqrt{\ell+1}$, and hence, Theorem \ref{pstchartw} yields signless Laplacian PST between $u$ and $v$ with minimum time $\tau=\frac{\pi}{2\sqrt{\ell+1}}=\frac{\pi}{\sqrt{2n}}$. Now, suppose $(2\ell+n+2)^2-8\ell n$ is a perfect square. By Theorem \ref{pstchartw}(3), signless Laplacian PST occurs between $u$ and $v$ if and only if 
\begin{equation}
\label{nu2}
\nu_2\left(2\ell-n+2+\sqrt{(2\ell+n+2)^2-8\ell n}\right)=\nu_2\left(2\ell-n+2-\sqrt{(2\ell+n+2)^2-8\ell n}\right).
\end{equation}
Using conjugation, (\ref{nu2}) yields $\nu_2\left((2\ell-n+2\pm \sqrt{(2\ell+n+2)^2-8\ell n})^2\right)=\nu_2(8n)$. Thus,
$\nu_2(n)$ is odd and we may write (\ref{nu2}) as
\begin{equation}
\label{nu3}
\nu_2\left(2\ell-n+2\pm\sqrt{(2\ell+n+2)^2-8\ell n}\right)=\frac{1}{2}\left(\nu_2(n)+3\right).
\end{equation}
Since $(2\ell+n+2)^2-8\ell n$ is a perfect square, we can write $8\ell n=4s(2\ell+n-s+2)$ for some integer $s$ such that $s(2\ell+n-s+2)$ is even, i.e., $n=\frac{s(2\ell-s+2)}{2\ell+s}$. Since $n$ is even, $s$ is also even. This allows us to write $2\ell+2-s=2(\ell+1-\frac{s}{2})$, and so we may write (\ref{nu3}) as
\begin{equation}
\label{dollar}
\nu_2\left(\ell+1-\frac{s}{2}\right)+1=\nu_2(n-s)=\frac{1}{2}\left(\nu_2(n)+1\right).
\end{equation}
Let $n=2^{\nu_2(n)}a$, $\ell=2^{\nu_2(\ell)}b$ and $s=2^{\nu_2(s)}c$ for odd $a$, $b$ and $c$. If $\nu_2(n)\leq \nu_2(s)$, then we can write $n-s=2^{\nu_2(n)}(a-2^{\nu_2(s)-\nu_2(n)}c)$. Thus, $\nu_2(n-s)= \nu_2(n)>\frac{1}{2}(\nu_2(n)+1)$ whenever $3\leq \nu_2(n)<\nu_2(s)$ while  $\nu_2(n-s)>\nu_2(n)\geq \frac{1}{2}(\nu_2(n)+1)$ whenever $1\leq \nu_2(n)=\nu_2(s)$. Both subcases contradict (\ref{dollar}), and this allows us to narrow down to the following cases.
\begin{itemize}
\item \small{Let $\nu_2(s)>1$ and $\nu_2(n)=1$ so that $\nu_2(n-s)=1$. If $\ell$ is even, then $\nu_2(\ell+1-\frac{s}{2})=0$. Otherwise, $\nu_2(\ell+1-\frac{s}{2})>0$. Thus, (\ref{dollar}) holds if and only if $\ell$ is even.}
\item Let $\nu_2(n)>\nu_2(s)$ so that $\nu_2(n)\geq 3$. Then $\nu_2(n-s)=\nu_2(s)$, and (\ref{dollar}) holds if and only if
\begin{equation}
\label{dollar1}
\nu_2\left(\ell+1-\frac{s}{2}\right)+1=\nu_2(s)=\frac{1}{2}(\nu_2(n)+1)\geq 2.
\end{equation}
If $\ell$ is even, then $\nu_2(\ell+1-\frac{s}{2})=0$, a contradiction to (\ref{dollar1}). Now, suppose $\ell$ is odd and let $\ell+1=2^{\nu_2(\ell+1)}d$. We have the following subcases.
\begin{itemize}
\item \small{If $\nu_2(\ell+1)>\nu_2(s)-1$, then $\nu_2(\ell+1-\frac{s}{2})=\nu_2(s)-1$, and so (\ref{dollar1}) holds if and only if $\nu_2(s)=\frac{1}{2}(\nu_2(n)+1)$.}
\item If $\nu_2(\ell+1)=\nu_2(s)-1$, then $\nu_2(\ell+1-\frac{s}{2})>\nu_2(s)-1$, and so (\ref{dollar1}) fails.
\item If $\nu_2(\ell+1)<\nu_2(s)-1$, then $\nu_2(\ell+1-\frac{s}{2})=\nu_2(\ell+1)$, and so (\ref{dollar1}) fails.
\end{itemize}
\end{itemize}
Combining the two cases above, we get PST occurring  between $u$ and $v$ whenever $(2\ell+n+2)^2-8\ell n$ is a perfect square if and only if $n=\frac{s(2\ell-s+2)}{2\ell+s}$ for some integer $s$, and either (i) $\nu_2(s)>1$, $\nu_2(n)=1$ and $\ell$ is even or (ii) $\ell$ is odd, $\nu_2(n)>\nu_2(s)$ and $\nu_2(\ell+1)>\nu_2(s)-1$. Finally, we invoke Theorem \ref{pstchartw} to get the minimum time that signless Laplacian PST occurs for the case $\ell>0$ and $(2\ell+n+2)^2-8\ell n$ is a perfect square.
\end{itemize}
Thus, we have covered all cases.
\end{proof}

Next, we examine the case of connected double cones.

\begin{theorem}
\label{condc}
Let $Y$ be a $\ell$-regular graph on $n\geq 1$ vertices with $\ell\neq n-1$. The following hold.
\begin{enumerate}
\item Adjacency perfect state transfer occurs between the apexes of $K_2\vee Y$ if and only if $n=\frac{s(\ell-1+s)}{2}$ for some integer $s$ satisfying $\nu_2(\ell+3)>\nu_2(s-2)\geq 1$ (so that $\ell$ is odd and $s$ is even). In particular, if $n=\frac{s(\ell-1+s)}{2}$ and we let $\ell+1=2^{\nu_2(\ell+1)}a$ and $s=2^{\nu_2(s)}b$ for some odd $a$ and $b$, then the condition $\nu_2(\ell+3)>\nu_2(s-2)\geq 1$ holds if and only if $\nu_2(\ell+1)=1$ and either (i) $\nu_2(s)=1$ and $\nu_2(a+1)> \nu_2(b-1)$, or (ii) $\nu_2(s)>1$. Moreover, the minimum time that perfect state transfer occurs is $\tau=\frac{\pi}{g}$, where $g=\operatorname{gcd}(\ell+s+1,s-2)$.
\item Signless Laplacian perfect state transfer occurs between the apexes of $K_2\vee Y$ if and only if the following conditions hold
\begin{enumerate}
\item $n=2\ell+4$; or
\item $n=\frac{s(2\ell+s)}{s+2}$ for some integer $s$ and either
\begin{enumerate}
\item $\ell$ is even and $\nu_2(s)>1$; or
\item $\ell$ is odd, $n=2a$ and $s=2c$ for some odd $a$ and $c$, and either
\begin{enumerate}
\item $2\nu_2(c-1)=2\nu_2(a-1)-2=\nu_2(\ell+1)\geq 2$
\item $2\nu_2(c-1)=2\nu_2(a-1)=\nu_2(\ell+1)\geq 2$, and $\nu_2(x-y)=1$, where $x$ and $y$ are odd integers such that $c-1=2^{c-1}x$ and $a-1=2^{a-1}y$.
\end{enumerate}
\end{enumerate}
\end{enumerate}
Moreover, if $n=2\ell+4$ and $\ell+3$ is an odd perfect square, then the minimum time that perfect state transfer occurs is $\tau=\frac{\pi}{\sqrt{2(n+2)}}$. Otherwise, it is $\tau=\frac{\pi}{g}$, where $g=\operatorname{gcd}(2\ell-n+s+2,s-2)$.
\end{enumerate}
\end{theorem}
\begin{proof}
Let $u$ and $v$ be the apexes of $K_2\vee H$. By \cite[Corollary 6.9(1)]{Monterde2021}, $u$ and $v$ are strongly cospectral with respect to both $A$ and $Q$. By Lemmas \ref{alphabeta} and \ref{strcospchar}, we obtain $\sigma_{uv}^-(A)=\{-1\}$ and $\sigma_{uv}^-(Q)=\{n-1\}$. 

First, we prove the first statement. By Lemma \ref{esupp11}(2) yields $\sigma_u(A)=\{\lambda^{\pm},\theta\}$, where $\lambda^{\pm}=\frac{1}{2}(\ell+1\pm\sqrt{(\ell-1)^2+8n})$ and $\theta=-1$ in Theorem \ref{pstchartw}. Observe that we can write $\lambda^{\pm}=\theta+\frac{1}{2}(\ell+3\pm\sqrt{(\ell-1)^2+8n})$. Applying Theorem \ref{pstchartw}, adjacency PST occurs between $u$ and $v$ if and only if $(\ell-1)^2+8n$ is a perfect square and
\begin{equation}
\label{Bnu1}
\nu_2\left(\ell+3+\sqrt{(\ell-1)^2+8n}\right)=\nu_2\left(\ell+3-\sqrt{(\ell-1)^2+8n}\right).
\end{equation}
Note that $(\ell-1)^2+8n$ is a perfect square if and only if $8n=4s(\ell-1+s)$ for some integer $s$ such that $s(\ell-1+s)$ is even. Thus, we can write (\ref{Bnu1}) as $\nu_2(\ell+s+1)=\nu_2((\ell+3)+(s-2))=\nu_2(s-2)$, which is equivalent to $\nu_2(\ell+3)>\nu_2(s-2)$. Thus, $\ell$ is odd, and because $s(\ell-1+s)$ is even, $s$ is also even. Let $\ell+1=2^{\nu_2(\ell+1)}a$ and $s=2^{\nu_2(s)}b$ for some odd $a$ and $b$ so that $\ell+3=2(2^{\nu_2(\ell+1)-1}a+1)$ and $s-2=2(2^{\nu_2(s)-1}b-1)$. Using these two equations, one can show that (i) $\nu_2(\ell+3)>\nu_2(s-2)>1$ if and only if $\nu_2(\ell+1)=\nu_2(s)=1$ and $\nu_2(a+1)> \nu_2(b-1)$ and (ii) $\nu_2(\ell+3)>\nu_2(s-2)=1$ if and only if $\nu_2(\ell+1)=1$ and $\nu_2(s)>1$. Equivalently, $\nu_2(\ell+1)=1$ and either (i) $\nu_2(s)=1$ and $\nu_2(a+1)\geq \nu_2(b-1)$, or (ii) $\nu_2(s)>1$. Thus, 1 holds.

Next, we prove the second statement in a similar way we proved Theorem \ref{discdc}(2). By Corollary \ref{esupp33}(2a), $\sigma_u(Q)=\{\lambda^{\pm},\theta\}$, where $\lambda^{\pm}=\frac{1}{2}\left(2\ell+n+4\pm\sqrt{(2\ell+n+4)^2-8(2\ell+2+\ell n)}\right)$ and $\theta=n$ in Theorem \ref{pstchartw}. Observe that we can write $\lambda^{\pm}=\theta+\frac{1}{2}\left(2\ell-n+4\pm\sqrt{(2\ell-n)^2+8n}\right)$. Invoking Theorems \ref{join33}(2) and \ref{pstchartw}(2), we need either $n=2\ell+4$ or $(2\ell-n)^2+8n$ is a perfect square. Let us first look at what happens when $n=2\ell+4$. This yields $\lambda^{\pm}=\theta+\sqrt{2(n+2)}$, and so \ref{pstchartw}(3) gives us signless Laplacian PST. Now, suppose that $(2\ell-n)^2+8n$ is a perfect square. By Theorem \ref{pstchartw}(3), we get signless Laplacian PST between $u$ and $v$ if and only if 
\begin{equation}
\label{mu2}
\nu_2\left(2\ell-n+4+\sqrt{(2\ell-n)^2+8n}\right)=\nu_2\left(2\ell-n+4-\sqrt{(2\ell-n)^2+8n}\right).
\end{equation}
Again, by conjugation, (\ref{mu2}) yields $\nu_2\left((2\ell-n+4\pm \sqrt{(2\ell-n)^2+8n})^2\right)=\nu_2(16(\ell+1))$. Therefore,
$\nu_2(\ell+1)$ is even and we may write (\ref{mu2}) as
\begin{equation}
\label{mu3}
\nu_2\left(2\ell-n+4\pm\sqrt{(2\ell-n)^2+8n}\right)=\frac{1}{2}\nu_2(\ell+1)+2.
\end{equation}
Since $(2\ell-n)^2+8n$ is a perfect square, we can write $8n=4s(2\ell-n+s)$ for some integer $s$ such that $s(2\ell-n+s)$ is even, i.e., $n=\frac{s(2\ell+s)}{s+2}$. Thus, we may write (\ref{mu3}) as
\begin{equation}
\label{dolla}
\nu_2\left(2\ell-n+s+2\right)=\nu_2(s-2)=\frac{1}{2}\nu_2(\ell+1)+1\geq 1.
\end{equation}
If $s$ and $n$ have opposite parities, then $\nu_2\left(2\ell-n+s+2\right)=0$, a contradiction to (\ref{dolla}). Moreover, if $s$ is odd, then $\nu_2(s-2)=0$, again a contradiction to (\ref{dolla}), Thus, both $s$ and $n$ are even. Let $n=2^{\nu_2(n)}a$, $\ell=2^{\nu_2(\ell)}b$ and $s=2^{\nu_2(s)}c$ for odd $a$, $b$ and $c$. We have the following cases.
\begin{itemize}
\item \small{Let $\ell$ be even so that $\nu_2(\ell+1)=0$. By (\ref{dolla}), we get $\nu_2\left(2\ell-n+s+2\right)=\nu_2(s-2)=1$, which holds if and only if $\nu_2(s)>1$ and $\nu_2(n)>1$. As $n=\frac{s(2\ell+s)}{s+2}$, the condition $\nu_2(s)>1$ implies that $\nu_2(n)>1$. Thus, we get signless Laplacian PST in this case if and only if $\nu_2(s)>1$.}
\item Let $\ell$ be odd so that $\nu_2(\ell+1)\geq 2$. Note that $2\ell-n+s+2=2((\ell+1)+(s-n)/2)$. If $\nu_2(n)>\nu_2(s)=1$, then $\nu_2(2\ell-n+s+2)=1$, a contradiction to (\ref{dolla}). Moreover, if $\nu_2(s)>1$, then $\nu_2(s-2)=1$, again a contradiction to (\ref{dolla}). Thus, the only case left is $\nu_2(n)=\nu_2(s)=1$. Let $\nu_2(n)=\nu_2(s)=1$ so that $2\ell-n+s+2=2((\ell+1)+(c-a)/2)$ and $\nu_2(s-2)=2(c-1)$. Then we can write (\ref{dolla}) as
\begin{equation}
\label{dolla2}
\nu_2\left((\ell+1)+(c-a)/2\right)=\nu_2(c-1)=\frac{1}{2}\nu_2(\ell+1)\geq 1.
\end{equation}
\begin{itemize}
\item \small{If $\nu_2(\ell+1)\leq \nu_2(c-a)-1$, then $\nu_2\left((\ell+1)+(c-a)/2\right)\geq \nu_2(\ell+1)> \frac{1}{2}\nu_2(\ell+1)$ because $\nu_2(\ell+1)\geq 2$, a contradiction to (\ref{dolla2}).}
\item If $\nu_2(\ell+1)>\nu_2(c-a)-1$, then $\nu_2\left((\ell+1)+(c-a)/2\right)=\nu_2(c-a)-1$ and so we can write (\ref{dolla2}) as
\begin{equation}
\label{dolla3}
\nu_2(c-a)-1=\nu_2(c-1)=\frac{1}{2}\nu_2(\ell+1).
\end{equation}
If $a=1$, then $\nu_2(c-1)-1=\nu_2(c-1)$ by (\ref{dolla3}), a contradiction, and so $a\geq 3$. Since $c$ and $a$ are odd, we may write $c-1=2^{c-1}x$ and $a-1=2^{a-1}y$, where $x$ and $y$ are odd. Thus, $c-a=2^{c-1}x-2^{a-1}y$. If $\nu_2(c-1)<\nu_2(a-1)$, then $\nu_2(c-a)=\nu_2(c-1)$, and so (\ref{dolla3}) fails. If if $\nu_2(c-1)>\nu_2(a-1)$, then one checks that $\nu_2(c-a)-1=\nu_2(c-1)$ if and only if $\nu_2(c-1)=\nu_2(a-1)-1$. Finally, if $\nu_2(c-1)=\nu_2(a-1)$, then we get $c-a=2^{c-1}(x-y)$, and therefore, $\nu_2(c-a)-1=\nu_2(c-1)$ if and only if $\nu_2(x-y)=1$.
\end{itemize}
Combining the subcases above  for the case when $\ell$ is odd, we obtain signless Laplacian PST if and only if either (i) $\nu_2(c-1)=\nu_2(a-1)-1$ or (ii) $\nu_2(c-1)=\nu_2(a-1)$ and $\nu_2(x-y)=1$. In both cases, (\ref{dolla3}) yields $2\nu_2(c-1)=\nu_2(\ell+1)$, and so we may write conditions (i) and (ii) as (i) $2\nu_2(c-1)=2\nu_2(a-1)-2=\nu_2(\ell+1)\geq 2$ and (ii) $2\nu_2(c-1)=2\nu_2(a-1)=\nu_2(\ell+1)\geq 2$ and $\nu_2(x-y)=1$, respectively.
\end{itemize}
Finally, applying Theorem \ref{pstchartw} yields the minimum PST time.
\end{proof}

It is helpful to note that Theorem \ref{condc} implies that $n$ must be even for adjacency and signless Laplacian PST to occur in $K_2\vee Y$. By virtue of Theorem \ref{discdc}, the same holds for $O_2\vee Y$ with respect to the signless Laplacian matrix, as well as the adjacency matrix provided that $\ell>0$. This observation rules out values of $n$ for which adjacency and signless Laplacian PST can occur in double cones on regular graphs. 

One can show that the conditions in Theorem \ref{discdc}(1b) extend work by Angeles-Canul et al.\ \cite[Corollaries 13, 15]{Angeles-Canul2010}, as their work does not have the additional parameter $s$. At the same time as \cite{Monterde}, Coutinho and Godsil provided a characterization of adjacency PST in double cones \cite[Lemmas 12.4.1, 12.4.2]{Coutinho2021}. Theorem \ref{discdc}(1) and Theorem \ref{condc}(1) coincide with their results, although the proofs herein highlight the utility of the theory we have developed for twin vertices. 

For case of signless Laplacian PST in disconnected double cones, Alvir et al.\ showed using equitable partitions that if $Y$ is an $(m-1)$-regular graph with $2m$ vertices, then $O_2\vee Y$ has PST between its apexes \cite[Theorem 7]{Alvir2016}. These are precisely the graphs described in Theorem \ref{discdc}(2b) with $\ell=m-1$ and $n=2m$, while those that satisfy Theorem \ref{discdc}(2a) are $K_{2,n}$ with $n\equiv 0$ (mod 4). To illustrate Theorem \ref{discdc}(2c), take $s=2(\ell-\sqrt{\ell})$ such that $\ell>0$ is an even perfect square to get an infinite family of disconnected double cones on $\ell$-regular graphs with $n=2\ell-2$ vertices that exhibit signless Laplacian PST with minimum time $\tau=\frac{\pi}{2}$. To the best of our knowledge, there are no previously published results on signless Laplacian PST in connected double cones. To generate one such family of graphs that exhibit signless Laplacian PST, one may use Theorem \ref{condc}(2a). Another way is by using Theorem \ref{condc}(2b). Indeed, we get one by taking $s=2\sqrt{\ell}$ such that $\ell>0$ is an even perfect square, and this family exhibits signless Laplacian PST with minimum time $\tau=\frac{\pi}{2}$.

We also remark that for the Laplacian case, PST in double cones over graphs with at least one vertex was fully characterized by Alvir et al.\ in \cite{Alvir2016} (see Corollary 5 for $O_2\vee Y$ and Corollary 6 for $K_2\vee Y$). In particular, they showed that if $Y$ is a graph on $n\geq 1$ vertices, then $O_2\vee Y$ has PST between its apexes if and only if $n\equiv 2$ (mod 4) with minimum PST time $\tau=\frac{\pi}{2}$. In contrast, $K_2\vee Y$ has no PST.

We end this section with an interesting observation about phase factors in signless Laplacian PST. Let $u$ and $v$ be the apexes of $X\vee Y$, where $X\in\{O_2,K_2\}$ and $Y$ be an $\ell$-regular graph on $n$ vertices graph such that $\ell\neq n-1$ whenever $X=K_2$. From Corollary \ref{esupp33}, we know that $\sigma_u(Q)=\{\theta,\lambda^{\pm}\}$, where $\theta=n$ and $\lambda^{\pm}$ depends on whether $X=O_2$ or $X=K_2$. In both cases, we know that $u$ and $v$ are strongly cospectral, and invoking Lemmas \ref{alphabeta} and \ref{strcospchar} gives us $\sigma_{uv}^-(Q)=\{\theta\}$. Thus, (\ref{haha1}) yields
\begin{equation*}
e^{it\lambda^+}=e^{it\lambda^-}=-e^{it\theta}
\end{equation*}
Now, if $\ell>0$ and $n=2\ell+2$, then Theorem \ref{discdc}(2b) implies that signless Laplacian PST occurs between the apexes of $O_2\vee Y$ with minimum PST time $\tau=\frac{\pi}{\sqrt{2n}}$. Thus, if $t=\tau s$ for any odd integer $s$ and $2n$ is not a perfect square, then $t\theta=\frac{\pi s\sqrt{2n}}{2}$ is irrational, and consequently, the phase factor $\gamma=-e^{it\theta}$ for signless Laplacian PST is not a root of unity. The same can be said for $K_2\vee Y$ whenever $n=2\ell+4$ and $2(n+2)$ is not a perfect square, in which case the minimum PST time $\tau=\frac{\pi}{\sqrt{2(n+2)}}$. This observation complements a remark of Coutinho and Godsil in \cite[Section 12.5, pp. 220]{Coutinho2021} which states that the phase factor for all currently known cases of adjacency PST in simple unweighted graphs is a root of unity. We conjecture that this is not true for unweighted graphs in general. For weighted graphs with possible loops, it is easy to see that adding $\eta I$ to $M$ for some $\eta\notin\mathbb{Q}\pi$ introduces a phase factor that is not a root of unity. Finally, for the Laplacian case, we know that $0\in \sigma_{uv}^+(L)$ for any two strongly cospectral vertices $u$ and $v$, and so the phase factor for all cases of PST in simple unweighted graphs is $\gamma=1$. This result can be shown to extend to all simple graphs with positive edge weights.

\section{Pretty good state transfer}\label{secPGST}

The following result characterizes twin vertices that exhibit PGST, which is an immediate consequence of \cite[Lemma 2.2]{Kempton2017a}.

\begin{theorem}
\label{pgsttw}
Let $T=\{u,v\}$ be a set of twins in $X$ and $\sigma_u(M)=\{\theta,\lambda_1,\ldots,\lambda_r\}$, where $\theta$ is given in (\ref{adjalpha}). Then pretty good state transfer occurs between $u$ and $v$ if and only if the following conditions hold.
\begin{enumerate}
\item $u$ and $v$ are strongly cospectral with $\sigma_{uv}^+(M)=\{\lambda_1,\ldots,\lambda_r\}$ and $\sigma_{uv}^-(M)=\{\theta\}$.
\item If $m_j$ are integers such that
\begin{equation}
\sum_{j}m_j(\lambda_j-\theta)=0
\end{equation}
then 
\begin{equation}
\sum_jm_j\ \text{is even}.
\end{equation}
\end{enumerate}
\end{theorem}

By Theorem \ref{pgsttw}, if $u$ and $v$ are strongly cospectral twins with $\theta=0$ and $\sigma_{uv}^+(M)$ is a linearly independent set over $\mathbb{Q}$, then PGST occurs between $u$ and $v$.

PGST is a relaxation of PST, and it is known that these two quantum phenomena are equivalent for periodic vertices (see for instance, \cite{Pal2019}). Thus, to distinguish vertices that exhibit PGST but not PST, we say that \textit{proper pretty good state transfer} occurs between two vertices if PGST occurs between them and they are not periodic. The following result characterizes the double cones that exhibit proper PGST thereby providing a family of examples of graphs exhibiting proper PGST.

\begin{theorem}
\label{discdcpgst}
Let $Y$ be a simple unweighted graph on $n\geq 1$ vertices. The following hold.
\begin{enumerate}
\item Proper pretty good state transfer does not occur between the apexes of $O_2\vee Y$ and $K_2\vee Y$ with respect to the Laplacian matrix.
\item Suppose $Y$ is $\ell$-regular. The following hold with respect to the adjacency matrix.
\begin{enumerate}
\item Proper pretty good state transfer occurs between the apexes of $O_2\vee Y$ if and only if $\ell>0$ and $\ell^2+8n$ is not a perfect square.
\item Proper pretty good state transfer occurs between the apexes of $K_2\vee Y$ if and only if $\ell\neq n-1$ and $(\ell-1)^2+8n$ is not a perfect square.
\end{enumerate}
\item Suppose $Y$ is $\ell$-regular. The following hold with respect to the signless Laplacian matrix.
\begin{enumerate}
\item Proper pretty good state transfer occurs between the apexes of $O_2\vee Y$ if and only if $\ell>0$, $n\neq 2\ell+2$ and $(2\ell+n+2)^2-8\ell n$ is not a perfect square.
\item Proper pretty good state transfer occurs between the apexes of $K_2\vee Y$ if and only if $\ell\neq n-1$, $n\neq 2\ell+4$ and $(2\ell+n+4)^2-8(2\ell+\ell n+2)$ is not a perfect square.
\end{enumerate}
\end{enumerate}
\end{theorem}

\begin{proof}
Let $u$ and $v$ be apexes of both $O_2\vee Y$ and $K_2\vee Y$. By Theorem \ref{join22}(1), we know that $u$ and $v$ are both periodic, and so proper PGST does not occur between them. Thus, the first statement holds. Next, we know from the proofs of Theorems \ref{discdc} and \ref{condc} that the apexes of $O_2\vee Y$ are adjacency and signless Laplacian strongly cospectral, while the apexes of $K_2\vee Y$ are adjacency and signless Laplacian strongly cospectral if and only if $\ell\neq n-1$ (i.e., $H\neq K_n$). Thus, to prove the second, we examine two cases.
\begin{itemize}
\item \small{In $O_2\vee Y$, Lemma \ref{esupp11}(2) gives us $\sigma_u(A)=\{\lambda^{\pm},\theta\}$, where $\lambda^{\pm}=\frac{1}{2}\left(\ell\pm\sqrt{\ell^2+8n}\right)$ and $\theta=0$. By Theorem \ref{join11}(2), $u$ and $v$ are periodic if and only if either $\ell=0$ or $\ell^2+8n$ is a perfect square. Thus, if proper PGST occurs between $u$ and $v$, then $\ell>0$ and $\ell^2+8n$ is not a perfect square, in which case $\sigma_{uv}^+(A)=\{\lambda^{\pm}\}$ is a linearly independent set over $\mathbb{Q}$.}
\item Meanwhile, in $K_2\vee Y$, Lemma \ref{esupp11}(2) yields $\sigma_u(A)=\{\lambda^{\pm},\theta\}$, where $\lambda^{\pm}=\frac{1}{2}\left(\ell+1\pm\sqrt{(\ell-1)^2+8n}\right)$ and $\theta=-1$. By Theorem \ref{join11}(2), $u$ and $v$ are periodic if and only if $(\ell-1)^2+8n$ is a perfect square. Hence, if proper PGST occurs between $u$ and $v$, then it must be that $(\ell-1)^2+8n$ is not a perfect square, in which case $\sigma_{uv}^+(A)=\{\lambda^{\pm}\}$ is a linearly independent set over $\mathbb{Q}$. 
\end{itemize}
In both cases, applying Theorem \ref{pgsttw} yields proper PGST between $u$ and $v$. Finally, let us prove the last statement. Let $\theta=n$. We know from Corollary \ref{esupp33}(2) that $\sigma_u(Q)=\{\lambda^{\pm},\theta\}$ in $O_2\vee Y$, where $\lambda^{\pm}=\theta+\frac{1}{2}\left(2\ell-n+4\pm\sqrt{\Delta}\right)$ and $\Delta=(2\ell+n+2)^2-8\ell n$, while $\sigma_u(Q)=\{\lambda^{\pm},\theta\}$ in $K_2\vee Y$, where $\lambda^{\pm}=\theta+\frac{1}{2}\left(2\ell-n+2\pm\sqrt{\Delta}\right)$ and $\Delta=(2\ell+n+4)^2-8(2\ell+\ell n+2)$. Following the same argument above yields the desired result.
\end{proof}

We illustrate Theorem \ref{discdcpgst} using a complete graph minus an edge.

\begin{example}
\label{knminuse}
Let $m\geq 4$, and consider the complete graph minus an edge $K_m\backslash e$, where $u$ and $v$ are the vertices joining the missing edge $e$. Note that $K_m\backslash e=O_2\vee K_{m-2}$, where $K_{m-2}$ is an $(m-3)$-regular graph on $m-2$ vertices. Let $k=m-3$ and $n=m-2$. We make the following observations about $K_m\backslash e$.
\begin{itemize}
\item \small{Since $m-3>0$ and $k^2+8n=(m+1)^2-8$ is not a perfect square for all $m\geq 4$, Theorem \ref{discdcpgst}(2a) yields proper adjacency PGST between $u$ and $v$. Indeed, since $\sigma_u(A)=\{0,m-3\pm\sqrt{(m+1)^2-8}\}$, Theorem \ref{ratiocon} implies that $u$ is not periodic, and so adjacency PST cannot occur between them. The authors in \cite{Casaccino2009} conjectured based on numerical observations that adjacency PST occurs between the $u$ and $v$ for all $m\geq 4$. But as we have shown, this is not the case, and what the authors observed was in fact proper PGST. Indeed, the paper was published in 2009, and PGST was not formally introduced until 2012.}
\item We show that Laplacian PST occurs in $K_m\backslash e$. By Lemma \ref{esupp22}(2), $\sigma_u(L)=\{0,m,\theta\}$, where $\theta=m-2$, and note that we can write $0=\theta+(m-2)$ and $m=\theta+2$. Since $\nu_2(2)=\nu_2(m-2)$ if and only if $m-2=2q$ for some odd $q$, Theorem \ref{pstchartw} yields Laplacian PST between $u$ and $v$ if and only if $m\equiv 0$ (mod 4). In particular, the minimum time that Laplacian PST occurs between $u$ and $v$ is $\tau=\frac{\pi}{2}$. Thus, the converse of \cite[Theorem 5]{Bose2008}, which states that Laplacian PST occurs between the two non-adjacent vertices of $K_m\backslash e$ whenever $m\equiv 0$ (mod 4), is also true. A more general result can be found in \cite[Theorem 2]{Severini}, which characterizes Laplacian PST in threshold graphs, of which $K_m\backslash e$ is an example.
\item For the signless Laplacian case, one checks that $(2k+n+2)^2-8kn=(m+2)^2-16$. However, since $(m+2)^2-16$ is not a perfect square for all $m\geq 4$, invokingTheorem \ref{discdcpgst}(2a) gives us signless Laplacian proper PGST between $u$ and $v$.
\end{itemize}
\end{example}

We end this section by combining Theorems \ref{join11} and \ref{join33}, Theorems \ref{discdc} and \ref{condc}, and Theorem \ref{discdcpgst} to provide a plethora of cases when PST, PGST or periodicity occurs, or cannot occur, in double cones.

\begin{example}
Let $Y$ be an $\ell$-regular graph on $n$ vertices.
\begin{enumerate}
\item For the adjacency case, we have the following.
\begin{enumerate}
\item Let $\ell=4$, and let $X=O_2\vee Y$.
\begin{enumerate}
\item If $n=\frac{1}{2}s(s+4)$ for some integer $s$ satisfying $\nu_2(s)=1$, then the apexes of $X$ exhibit PST by Theorem \ref{discdc}(1) with the minimum PST time $\tau=\frac{\pi}{2}$.
\item If $n=\frac{1}{2}s(s+4)$ for some integer $s$ satisfying $\nu_2(s)>1$, then the apexes of $X$ are periodic by Theorem \ref{join11}(2b) with $\rho=\pi$, but they do not exhibit PST by Theorem \ref{discdc}(1).
\item If $n\neq \frac{1}{2}s(s+4)$ for any integer $s$, then the apexes of $X$ exhibit proper PGST by Theorem \ref{discdcpgst}(2a).
\end{enumerate}
\item Let $\ell=1$, and let $X=K_2\vee Y$.
\begin{enumerate}
\item If $n=\frac{s^2}{2}$ for some integer $s$ satisfying $\nu_2(s)>1$, then the apexes of $X$ exhibit PST by Theorem \ref{condc}(1) with the minimum PST time is $\frac{\pi}{2}$.
\item If $n=\frac{s^2}{2}$ for some integer $s$ satisfying $\nu_2(s)=1$, then the apexes of $X$ are periodic by Theorem \ref{join11}(2b) with $\rho=\pi$, but they do not exhibit PST by Theorem \ref{condc}(1).
\item If $n\neq \frac{s^2}{2}$ for any integer $s$, then the apexes of $X$ exhibit proper PGST by Theorem \ref{discdcpgst}(2b).
\end{enumerate}
\end{enumerate}
\item For the signless Laplacian case, we have the following.
\begin{enumerate}
\item Let $\ell>0$ and $n=2\ell-2$, and let $X=O_2\vee Y$.
\begin{enumerate}
\item If $\ell$ is an even perfect square, then the apexes of $X$ exhibit PST by Theorem \ref{discdc}(2c) with minimum PST time $\tau=\frac{\pi}{2}$.
\item If $\ell$ is an odd perfect square, then the apexes of $X$ are periodic by Theorem \ref{join33}(2) with $\rho=\pi$, but they do not exhibit PST by Theorem \ref{discdc}(2c).
\item If $\ell$ is not a perfect square, then the apexes of $X$ exhibit proper PGST by Theorem \ref{discdcpgst}(3a).
\end{enumerate}
\item Let $\ell>0$ and $n=2\ell$, and let $X=K_2\vee Y$.
\begin{enumerate}
\item If $\ell$ is an even perfect square, then the apexes of $X$ exhibit PST by Theorem \ref{condc}(2c) with minimum PST time $\tau=\frac{\pi}{2}$.
\item If $\ell$ is an odd perfect square, then the apexes of $X$ are periodic by Theorem \ref{join33}(1) with $\rho=\pi$, but they do not exhibit PST by Theorem \ref{condc}(2c).
\item If $\ell$ is not a perfect square, then the apexes of $X$ exhibit proper PGST by Theorem \ref{discdcpgst}(3b).
\end{enumerate}
\end{enumerate}
\end{enumerate}
\end{example}

\section*{Acknowledgments}
H.M.\ is supported by the University of Manitoba Faculty of Science and Faculty of Graduate Studies. S.K.\ is supported by NSERC Discovery Grant RGPIN-2019-05408. S.P.\ is supported by NSERC Discovery Grant number 1174582, the Canada Foundation for Innovation (CFI) grant number 35711, and the Canada Research Chairs (CRC) Program grant number 231250. We thank the referees for their comments and suggestions that helped improve this paper.

%
%

\bibliographystyle{alpha}
\bibliography{mybibfile}

\begin{bibdiv}
\begin{biblist}

\bib{Angeles-Canul2009}{article}{
      author={Angeles-Canul, Ricardo},
      author={Norton, Rachael},
      author={Opperman, Michael},
      author={Paribello, Christopher},
      author={Russell, Matthew},
      author={Tamon, Christino},
        date={2009},
     journal={International Journal of Quantum Information},
      volume={07},
      number={08},
       pages={1429\ndash 1445},
}

\bib{Angeles-Canul2010}{article}{
      author={Angeles-Canul, Ricardo},
      author={Norton, Rachael},
      author={Opperman, Michael},
      author={Paribello, Christopher},
      author={Russell, Matthew},
      author={Tamon, Christino},
       title={{Perfect state transfer, integral circulants, and join of
  graphs}},
        date={2010},
        ISSN={15337146},
     journal={Quantum Information and Computation},
      volume={10},
      number={3-4},
       pages={0325\ndash 0342},
}

\bib{Alvir2016}{article}{
      author={Alvir, Rachael},
      author={Dever, Sophia},
      author={Lovitz, Benjamin},
      author={Myer, James},
      author={Tamon, Christino},
      author={Xu, Yan},
      author={Zhan, Hanmeng},
       title={{Perfect state transfer in Laplacian quantum walk}},
        date={2016},
        ISSN={15729192},
     journal={Journal of Algebraic Combinatorics},
      volume={43},
      number={4},
       pages={801\ndash 826},
         url={https://arxiv.org/abs/1409.5840},
}

\bib{Banchi2017}{article}{
      author={Banchi, Leonardo},
      author={Coutinho, Gabriel},
      author={Godsil, Chris},
      author={Severini, Simone},
       title={{Pretty good state transfer in qubit chains-The Heisenberg
  Hamiltonian}},
        date={2017},
        ISSN={00222488},
     journal={Journal of Mathematical Physics},
      volume={58},
      number={3},
       pages={32202},
         url={https://doi.org/10.1063/1.4978327},
}

\bib{Bose2008}{article}{
      author={Bose, Sougato},
      author={Casaccino, Andrea},
      author={Mancini, Stefano},
      author={Severini, Simone},
       title={{Communication in XYZ all-to-all quantum networks with a missing
  link}},
        date={2009},
        ISSN={02197499},
     journal={International Journal of Quantum Information},
      volume={7},
      number={4},
       pages={713\ndash 723},
         url={http://arxiv.org/abs/0808.0748},
}

\bib{Bachman2011}{article}{
      author={Bachman, Rachel},
      author={Fredette, Eric},
      author={Fuller, Jessica},
      author={Landry, Michael},
      author={Opperman, Michael},
      author={Tamon, Christino},
      author={Tollefson, Andrew},
       title={{Perfect state transfer on quotient graphs}},
        date={2012},
        ISSN={15337146},
     journal={Quantum Information and Computation},
      volume={12},
      number={3-4},
       pages={293\ndash 313},
         url={http://arxiv.org/abs/1108.0339},
}

\bib{Bose:Quantum}{article}{
      author={Bose, Sougato},
       title={{Quantum communication through an unmodulated spin chain}},
        date={2003},
     journal={Physical Review Letters},
      volume={91},
      number={20},
       pages={207901},
}

\bib{Basic2009}{article}{
      author={Ba{\v{s}}i{\'{c}}, Milan},
      author={Petkovi{\'{c}}, Marko},
      author={Stevanovi{\'{c}}, Dragan},
       title={{Perfect state transfer in integral circulant graphs}},
        date={2009},
        ISSN={08939659},
     journal={Applied Mathematics Letters},
      volume={22},
      number={7},
       pages={1117\ndash 1121},
}

\bib{Christandl2005}{article}{
      author={Christandl, Matthias},
      author={Datta, Nilanjana},
      author={Dorlas, Tony~C.},
      author={Ekert, Artur},
      author={Kay, Alastair},
      author={Landahl, Andrew~J.},
       title={{Perfect transfer of arbitrary states in quantum spin networks}},
        date={2005},
        ISSN={10502947},
     journal={Physical Review A - Atomic, Molecular, and Optical Physics},
      volume={71},
      number={3},
}

\bib{Christandl:PSTonHypercubes}{article}{
      author={Christandl, Matthias},
      author={Datta, Nilanjana},
      author={Ekert, Artur},
      author={Landahl, Andrew~J.},
       title={{Perfect state transfer in quantum spin networks}},
        date={2004},
     journal={Physical Review Letters},
      volume={92},
      number={18},
       pages={187902},
}

\bib{Cheung2011}{article}{
      author={Cheung, Wang~Chi},
      author={Godsil, Chris},
       title={{Perfect state transfer in cubelike graphs}},
        date={2011},
        ISSN={00243795},
     journal={Linear Algebra and Its Applications},
      volume={435},
      number={10},
       pages={2468\ndash 2474},
}

\bib{Chen2019PairST}{article}{
      author={Chen, Qiuting},
      author={Godsil, Chris~D.},
       title={Pair state transfer},
        date={2019},
     journal={Quantum Inf. Process.},
      volume={19},
       pages={321},
}

\bib{Coutinho2021}{article}{
      author={Coutinho, Gabriel},
      author={Godsil, Chris},
       title={{Graph Spectra and Continuous Quantum Walks}},
        date={2021},
      eprint={https:
  //www.math.uwaterloo.ca/~cgodsil/quagmire/pdfs/GrfSpc3.pdf},
}

\bib{Coutinho2015}{article}{
      author={Coutinho, Gabriel},
      author={Godsil, Chris},
      author={Guo, Krystal},
      author={Vanhove, Frederic},
       title={{Perfect state transfer on distance-regular graphs and
  association schemes}},
        date={2015},
        ISSN={00243795},
     journal={Linear Algebra and Its Applications},
      volume={478},
       pages={108\ndash 130},
}

\bib{Coutinho2014a}{article}{
      author={Coutinho, Gabriel},
      author={Liu, Henry},
       title={{No Laplacian perfect state transfer in trees}},
        date={2015},
        ISSN={08954801},
     journal={SIAM Journal on Discrete Mathematics},
      volume={29},
      number={4},
       pages={2179\ndash 2188},
         url={https://arxiv.org/abs/1408.2935},
}

\bib{Casaccino2009}{article}{
      author={Casaccino, Andrea},
      author={Lloyd, Seth},
      author={Mancini, Stefano},
      author={Severini, Simone},
       title={{Quantum state transfer through a qubit network with energy
  shifts and fluctuations}},
        date={2009},
        ISSN={02197499},
     journal={International Journal of Quantum Information},
      volume={7},
      number={8},
       pages={1417\ndash 1427},
         url={http://arxiv.org/abs/0904.4510},
}

\bib{Coutinho2014}{article}{
      author={Coutinho, Gabriel},
       title={{Quantum State Transfer in Graphs}},
        date={2014},
     journal={PhD. Dissertation},
}

\bib{Cao2020b}{article}{
      author={Cao, Xiwang},
      author={Wang, Dandan},
      author={Feng, Keqin},
       title={{Pretty good state transfer on Cayley graphs over dihedral
  groups}},
        date={2020},
        ISSN={0012365X},
     journal={Discrete Mathematics},
      volume={343},
      number={1},
       pages={111636},
}

\bib{Eisenberg2019}{article}{
      author={Eisenberg, Or},
      author={Kempton, Mark},
      author={Lippner, Gabor},
       title={{Pretty good quantum state transfer in asymmetric graphs via
  potential}},
        date={2019},
        ISSN={0012365X},
     journal={Discrete Mathematics},
      volume={342},
      number={10},
       pages={2821\ndash 2833},
}

\bib{Fan2013}{article}{
      author={Fan, Xiaoxia},
      author={Godsil, Chris},
       title={{Pretty good state transfer on double stars}},
        date={2013},
        ISSN={00243795},
     journal={Linear Algebra and Its Applications},
      volume={438},
      number={5},
       pages={2346\ndash 2358},
}

\bib{Farhi1998}{article}{
      author={Farhi, Edward},
      author={Gutmann, Sam},
       title={{Quantum computation and decision trees}},
        date={1998},
        ISSN={10941622},
     journal={Physical Review A - Atomic, Molecular, and Optical Physics},
      volume={58},
      number={2},
       pages={915\ndash 928},
}

\bib{Ge2011}{article}{
      author={Ge, Yang},
      author={Greenberg, Benjamin},
      author={Perez, Oscar},
      author={Tamon, Christino},
       title={{Perfect state transfer, graph products and equitable
  partitions}},
        date={2011},
        ISSN={02197499},
     journal={International Journal of Quantum Information},
      volume={9},
      number={3},
       pages={823\ndash 842},
         url={http://arxiv.org/abs/1009.1340},
}

\bib{Godsil2012}{article}{
      author={Godsil, Chris},
      author={Kirkland, Stephen},
      author={Severini, Simone},
      author={Smith, Jamie},
       title={{Number-theoretic nature of communication in quantum spin
  systems}},
        date={2012},
        ISSN={00319007},
     journal={Physical Review Letters},
      volume={109},
      number={5},
       pages={3\ndash 6},
}

\bib{Godsil2011}{article}{
      author={Godsil, Chris},
       title={{Periodic graphs}},
        date={2011},
        ISSN={10778926},
     journal={Electronic Journal of Combinatorics},
      volume={18},
      number={1},
       pages={1\ndash 15},
}

\bib{Godsil2012a}{article}{
      author={Godsil, Chris},
       title={{State transfer on graphs}},
        date={2012},
        ISSN={0012365X},
     journal={Discrete Mathematics},
      volume={312},
      number={1},
       pages={129\ndash 147},
}

\bib{Godsil2010}{article}{
      author={Godsil, Chris},
       title={{When can perfect state transfer occur?}},
        date={2012},
        ISSN={10813810},
     journal={Electronic Journal of Linear Algebra},
      volume={23},
       pages={877\ndash 890},
         url={http://arxiv.org/abs/1011.0231},
}

\bib{Godsil:AlgebraicGraph}{book}{
      author={Godsil, Chris},
      author={Royle, Gordon},
       title={{Algebraic Graph Theory}},
      series={Graduate Texts in Mathematics},
   publisher={Springer Verlag, New York},
        date={2001},
      volume={207},
}

\bib{Godsil2017}{article}{
      author={Godsil, Chris},
      author={Smith, Jamie},
       title={{Strongly cospectral vertices}},
        date={2017},
     journal={arXiv:1709.07975},
         url={http://arxiv.org/abs/1709.07975},
}

\bib{Horn:MatrixAnalysis}{book}{
      author={Horn, Roger},
      author={Johnson, Charles},
       title={{Matrix Analysis}},
     edition={Second Edition},
   publisher={Cambridge University Press, Cambridge},
        date={2013},
}

\bib{horn94}{book}{
      author={Horn, Roger},
      author={Johnson, Charles},
       title={{Topics in Matrix Analysis}},
   publisher={Cambridge University Press},
     address={Cambridge, New York},
        date={1994},
        ISBN={0521467136},
}

\bib{Johnston2017}{article}{
      author={Johnston, Nathaniel},
      author={Kirkland, Steve},
      author={Plosker, Sarah},
      author={Storey, Rebecca},
      author={Zhang, Xiaohong},
       title={{Perfect quantum state transfer using Hadamard diagonalizable
  graphs}},
        date={2017},
        ISSN={00243795},
     journal={Linear Algebra and Its Applications},
      volume={531},
       pages={375\ndash 398},
}

\bib{Kay2010}{article}{
      author={Kay, Alastair},
       title={{Perfect, efficient, state transfer and its application as a
  constructive tool}},
        date={2010},
        ISSN={02197499},
     journal={International Journal of Quantum Information},
      volume={8},
      number={4},
       pages={641\ndash 676},
}

\bib{Kay2011}{article}{
      author={Kay, Alastair},
       title={{Basics of perfect communication through quantum networks}},
        date={2011},
        ISSN={10502947},
     journal={Physical Review A - Atomic, Molecular, and Optical Physics},
      volume={84},
      number={2},
         url={http://arxiv.org/abs/1102.2338
  http://dx.doi.org/10.1103/PhysRevA.84.022337},
}

\bib{Kendon2003}{article}{
      author={Kendon, Vivien},
       title={{Quantum walks on general graphs}},
        date={2006},
        ISSN={02197499},
     journal={International Journal of Quantum Information},
      volume={4},
      number={5},
       pages={791\ndash 805},
         url={http://arxiv.org/abs/quant-ph/0306140
  http://dx.doi.org/10.1142/S0219749906002195},
}

\bib{Kempton2016}{article}{
      author={Kempton, Mark},
      author={Lippner, Gabor},
      author={Yau, Shing-Tung},
       title={{Perfect state transfer on graphs with a potential}},
        date={2017},
        ISSN={15337146},
     journal={Quantum Information and Computation},
      volume={17},
      number={3-4},
       pages={0303\ndash 0327},
      eprint={1611.02093},
         url={http://arxiv.org/abs/1611.02093},
}

\bib{Kempton2017a}{article}{
      author={Kempton, Mark},
      author={Lippner, Gabor},
      author={Yau, Shing-Tung},
       title={{Pretty good state transfer in graphs with an involution}},
        date={2017},
     journal={arXiv:1702.07000},
         url={https://arxiv.org/abs/1702.07000
  http://arxiv.org/abs/1702.07000},
}

\bib{Severini}{article}{
      author={Kirkland, Steve},
      author={Severini, Simone},
       title={Spin-system dynamics and fault detection in threshold networks},
        date={2011},
     journal={Phys. Rev. A},
      volume={83},
       pages={012310},
         url={https://link.aps.org/doi/10.1103/PhysRevA.83.012310},
}

\bib{Kendon2011}{article}{
      author={Kendon, Vivien},
      author={Tamon, Christino},
       title={{Perfect state transfer in quantum walks on graphs}},
        date={2011},
        ISSN={15461955},
     journal={Journal of Computational and Theoretical Nanoscience},
      volume={8},
      number={3},
       pages={422\ndash 433},
}

\bib{Li2021}{article}{
      author={Li, Yipeng},
      author={Liu, Xiaogang},
      author={Zhang, Shenggui},
      author={Zhou, Sanming},
       title={{Perfect state transfer in NEPS of complete graphs}},
        date={2021},
        ISSN={0166218X},
     journal={Discrete Applied Mathematics},
      volume={289},
       pages={98\ndash 114},
         url={http://arxiv.org/abs/2010.03198},
}

\bib{Monterde}{article}{
      author={Monterde, Hermie},
       title={{Quantum state transfer between twins in graphs}},
        date={2021},
     journal={MSc. Thesis},
      eprint={http://hdl.handle.net/1993/35937},
}

\bib{Monterde2021}{article}{
      author={Monterde, Hermie},
       title={Strong cospectrality and twin vertices in weighted graphs},
        date={2021},
     journal={The Electronic Journal of Linear Algebra},
      volume={38},
       pages={494\ndash 518},
}

\bib{Pal2019}{article}{
      author={Pal, Hiranmoy},
       title={Quantum state transfer on a class of circulant graphs},
        date={2019},
     journal={Linear and Multilinear Algebra},
      volume={69},
       pages={2527 \ndash  2538},
}

\bib{PAL2022112872}{article}{
      author={Pal, Hiranmoy},
       title={Laplacian state transfer on graphs with an edge perturbation
  between twin vertices},
        date={2022},
        ISSN={0012-365X},
     journal={Discrete Mathematics},
      volume={345},
      number={7},
       pages={112872},
  url={https://www.sciencedirect.com/science/article/pii/S0012365X22000784},
}

\bib{Pal2017}{article}{
      author={Pal, Hiranmoy},
      author={Bhattacharjya, Bikash},
       title={{Pretty good state transfer on circulant graphs}},
        date={2017},
        ISSN={10778926},
     journal={Electronic Journal of Combinatorics},
      volume={24},
      number={2},
}

\bib{Pal2017a}{article}{
      author={Pal, Hiranmoy},
      author={Bhattacharjya, Bikash},
       title={{Pretty good state transfer on some NEPS}},
        date={2017},
        ISSN={0012365X},
     journal={Discrete Mathematics},
      volume={340},
      number={4},
       pages={746\ndash 752},
         url={https://arxiv.org/abs/1604.08858},
}

\bib{PAL2017746}{article}{
      author={Pal, Hiranmoy},
      author={Bhattacharjya, Bikash},
       title={Pretty good state transfer on some neps},
        date={2017},
        ISSN={0012-365X},
     journal={Discrete Mathematics},
      volume={340},
      number={4},
       pages={746\ndash 752},
  url={https://www.sciencedirect.com/science/article/pii/S0012365X16303831},
}

\bib{VANBOMMEL2019}{article}{
      author={{van Bommel}, Christopher},
       title={{A complete characterization of pretty good state transfer on
  paths}},
        date={2019},
        ISSN={15337146},
     journal={Quantum Information and Computation},
      volume={19},
      number={7-8},
       pages={601\ndash 608},
}

\bib{Vinet2012}{article}{
      author={Vinet, Luc},
      author={Zhedanov, Alexei},
       title={{Almost perfect state transfer in quantum spin chains}},
        date={2012},
        ISSN={10502947},
     journal={Physical Review A - Atomic, Molecular, and Optical Physics},
      volume={86},
      number={5},
         url={http://arxiv.org/abs/1205.4680
  http://dx.doi.org/10.1103/PhysRevA.86.052319},
}

\end{biblist}
\end{bibdiv}
\end{document}